\documentclass[12pt,reqno]{amsart}
\usepackage[a4paper,top=2.5cm,bottom=2.5cm,left=2.5cm,right=2.5cm]{geometry}
\usepackage{setspace}
\usepackage{graphicx}
\usepackage{amsmath}
\usepackage{amsthm}
\usepackage{mathrsfs}
\usepackage{bbm}
\usepackage{amsfonts}
\usepackage{xcolor}
\colorlet{BLUE}{blue}
\colorlet{RED}{red}
\usepackage{cite}
\usepackage[normalem]{ulem}
\usepackage[colorlinks=true]{hyperref}
\hypersetup{
  linkcolor=blue,
  citecolor=green!40!black,
  urlcolor=green
}
\usepackage{tkz-euclide}  
\usepackage{pgfplots}
\pgfplotsset{compat=1.17}
\usepackage{amssymb}
\usepackage{fancyhdr}
\usepackage{lipsum}
\usepackage{comment}
\usepackage{euscript}

\newtheorem{theorem}{Theorem}[section]
 
\newtheorem{lemma}[theorem]{Lemma} 
\newtheorem{corollary}[theorem]{Corollary}
\newtheorem{example}[theorem]{Example}
\newtheorem{proposition}[theorem]{Proposition}
\newtheorem*{remark}{Remark}

\numberwithin{equation}{section}

\newcommand{\C}{\mathbb C}

\newcommand{\one}{\mathbbm 1}
\newcommand{\norm}[1]{\left\lVert #1\right\rVert}

\newcommand{\Sinf}{\mathbb N_0^{(\mathbb N)}}
\newcommand{\Ginf}{\mathbb Z^{(\mathbb N)}}

\usepackage[normalem]{ulem}
\usepackage{soul}

\begin{document}

\title[$L^p$-Khinchin sequences]{%
$L^p$-Khinchin sequences and transference principles} 

\begin{abstract}
A sequence
\((\lambda_n)_{n\ge1}\) of {distinct positive integers}
is called an \(L^p\)-Khinchin
sequence, where $1\le p\le\infty$, if
$
  n^{-1}\sum_{k=1}^n f(\lambda_k x)
  \longrightarrow
  \int_{\mathbb T}f\,d\mathfrak m
$
 a.e. $x\in\mathbb T$
for every \(f\in L^p(\mathbb T)\). 
When $1\le p<\infty$,   we characterize the \(L^p\)-Khinchin
sequences 
by a weak $(p,p)$ maximal inequality for the associated averaging operators.
We also establish a transference principle reducing the weak $(p,p)$ maximal inequality to weak $(p,p)$ maximal inequalities for  three model systems: translation along a lattice path in the infinite direct sum of the semigroup $\mathbb{N}_0$, translation along a lattice path in the infinite direct sum of the group $\mathbb{Z}$ and translation on the circle $\mathbb{T}$.
As consequences, 
we prove that $(2^n+1)_{n\ge1}$, the Fibonacci sequence, sequences of distinct primes and their partial products are \emph{not} $L^p$-Khinchin for any $1\le p<\infty$.
 We  study the special class of  multiplicative  sequences generated by
infinite words and we show that balanced infinite words generate $L^1$-Khinchin sequences, 
while typical random multiplicative sequences over at least two coprime integers are almost surely not $L^p$-Khinchin for any $1\le p<\infty$.
We also obtain a {Conze-type principle} for countably generated semigroup actions on a probability space.

\end{abstract}

\author{Aihua FAN}
\address{LAMFA, UMR 7352 CNRS, University of Picardie, 33 rue Saint Leu, 80039 Amiens, France}
\email{ai-hua.fan@u-picardie.fr}

\author{Shilei FAN}
\address{School of Mathematics and Statistics, and Key Lab NAA--MOE, Central China Normal University, Wuhan 430079, China}
\email{slfan@ccnu.edu.cn}

\author{Tongjin LI}
\address{School of Mathematics and Statistics \& Hubei Key Laboratory of Mathematical Sciences, Central China Normal University, Wuhan, 430079, China}
\email{litongjin.ccnu@qq.com}

\thanks{The authors were supported by NSFC (grants No. 12331004, No.  12231013 and No. 12571205)}
\makeatletter
\@namedef{subjclassname@2020}{\textup{2020} Mathematics Subject Classification}
\makeatother

\maketitle

\tableofcontents

\section{Introduction and main results}
\label{sec:introduction-main-results}

Let \(\Lambda=(\lambda_n)_{n\ge1}\) be a sequence of  distinct positive
integers, which is not assumed to be increasing. 
Let $\mathbb{T}$ be the additive group $\mathbb{R} / \mathbb{Z}$ 
identified with the unit interval $[0,1)$ and
endowed with Lebesgue measure \(\mathfrak m\).  It is well known that 
 for Lebesgue almost every
\(x\in\mathbb T\), the sequence \((\lambda_nx)_{n\ge1}\) is uniformly
distributed modulo one.  Hence, by Weyl's criterion \cite{Weyl1916}, for
every Riemann integrable function \(f\) on \(\mathbb T\),
\begin{equation}\label{KH-property}
\lim_{n\to\infty} \frac1n\sum_{k=1}^n f(\lambda_k x)
=
\int_{\mathbb T} f\,d\mathfrak m
\qquad
\text{for }\mathfrak m\text{-a.e. }x\in\mathbb T .
\end{equation}
Following
\cite{FFQQ}, we say that \(\Lambda\) is an \emph{\(L^p\)-Khinchin sequence},
for some  \(1\le p\le\infty\), if \eqref{KH-property} holds  for every \(f\in L^p(\mathbb T)\).

In 1923, Khinchin \cite{Khintchin}
conjectured that the whole sequence of natural numbers  \((n)_{n\ge1}\) is an
\(L^\infty\)-Khinchin sequence.  But, about fifty years later, Marstrand \cite{Marstrand} disproved this
conjecture by constructing a counter-example $f \in L^\infty(\mathbb{T})$ which is the indicator function of some  Borel set.
Later,  Bourgain \cite{Bourgain1988} and
Quas--Wierdl \cite{Quas_Wierdl2010} gave different proofs of Marstrand's result.  On
the other hand, for every integer \(q\ge2\), the geometric sequence
\((q^n)_{n\ge1}\) is \(L^1\)-Khinchin. This result is due to Raikov
\cite{Raikov}; Riesz \cite{Riesz} pointed out that it follows from
Birkhoff's ergodic theorem.
In \cite{Marstrand}, 
Marstrand also proved that the increasing enumeration of the Furstenberg semigroup
\(\{2^m3^n:m,n\ge0\}\) is \(L^\infty\)-Khinchin, and later Nair \cite{Nair1990}
proved the sequence is in fact \(L^1\)-Khinchin using Bewley's ergodic theorem for
\(\mathbb N_0^d\)-actions.

Finding \(L^p\)-Khinchin sequences among subsequences of
\((2^n)_{n\ge1}\) is closely related to the investigation of universal good or bad sequences in ergodic
theory.  Bourgain's results \cite{B1988,Bourgain1989} imply that
\((2^{n^2})_{n\ge1}\) is \(L^p\)-Khinchin for every \(p>1\), whereas  Buczolich and Mauldin's result  \cite{Buczolich2010} implies that it is not
\(L^1\)-Khinchin.  Let us also point out that a result of Rosenblatt \cite{Rosenblatt1991} implies  that
\((2^{2^n})_{n\ge1}\) is not \(L^\infty\)-Khinchin.

Distinguishing \(L^p\)-Khinchin sequences among subsequences of integers
seems a subtle problem.  In this paper we present some general principles
based on Banach's principle and its variants, and construct some
\(L^p\)-Khinchin sequences and some non-\(L^p\)-Khinchin sequences.  Some open questions asked in \cite{FFQQ} are answered. The
main objects of study are the following operators acting on
\(L^p(\mathbb T)\):
\begin{equation}\label{eq:Lambda_n}
\Lambda_n f(x)
:=
\frac1n\sum_{k=1}^n f(\lambda_k x),
\qquad n\ge1.
\end{equation}
We call them the {\em  Khinchin operators} associated with $\Lambda=(\lambda_n)$. 

In this paper, we concentrate on the study of $L^p$-Khinchin sequences for $1\le p<\infty$. $L^\infty$-Khinchin sequences are studied in \cite{FFL2}. 
The
\(L^\infty\)-case requires a different principle, due to Bellow and Jones, based on the continuity in
measure of the maximal operator. The transference principle in this case takes a different form.

We first observe that when \(1\le p<\infty\), the Banach principle and the Sawyer principle imply that
\(\Lambda\) is \(L^p\)-Khinchin if and only if the Khinchin operators satisfy the following weak \((p,p)\) maximal
inequality:
there exists a constant \(C>0\) such that for every
\(f\in L^p(\mathbb T)\) and every \(t>0\),
\begin{equation}\label{eq:WPMI}
   \mathfrak{m} \left\{ x \in \mathbb{T} :  \Lambda^*f(x) >t \right\} \le \frac{C}{t^p} \|f\|_p^p,
  \end{equation}
where
\(\Lambda^*f:=\sup_{n\ge1}|\Lambda_nf|\)
(see Theorem~\ref{thm:Lp-khinchin-maximal}). 

Secondly, following the idea of Calder\'on's transference principle (cf. \cite{Calderon1968}), we can transfer the  weak \((p,p)\) maximal
inequality \eqref{eq:WPMI} to more practical weak \((p,p)\) maximal
inequalities for simpler models. 
One model is the translation on the semigroup $\mathbb{N}_0^{(\mathbb N)}$ and the other is the  translation on the group $\mathbb{Z}^{(\mathbb N)}$. 
Here $\mathbb{N}_0:=\{0,1,2,\ldots\}$ denotes the additive semigroup of all non-negative integers and   $\mathbb{N}_0^{(\mathbb N)}$ denotes the countable direct sum of  $\mathbb{N}_0$, 
consisting of all sequences $x=(x_n)_{n\ge 1}$ with $x_n \in \mathbb{N}_0$ 
such that $x_n=0$ for all but finitely many $n$. Similarly, $\mathbb{Z}^{(\mathbb N)}$ denotes the countable direct sum of the additive group $\mathbb{Z}$. 
By the prime factorization theorem, 
the (multiplicative) semigroup $(\mathbb N,\cdot)$ is isomorphic to $(\mathbb N_0^{(\mathbb N)},+)$.
Indeed, let
\(\mathfrak p_1<\mathfrak p_2<\cdots\) be the primes. The canonical isomorphism
\(
\Psi: \mathbb N_0^{(\mathbb N)}\longrightarrow  \mathbb N
\) is given by 
\begin{equation}\label{eq:Phi}
\Psi(\mathbf{n})= 
\prod_{j\ge1}\mathfrak p_j^{n_j}
\quad\text{for }\mathbf{n}=(n_1,n_2,\ldots)\in\mathbb{N}_0^{(\mathbb N)}.
\end{equation}
We call the inverse of $\Psi$, denoted $\Phi$, the
{\em lifting map}
from $\mathbb{N}$ to
$\mathbb N_0^{(\mathbb N)}$.

For a given sequence $\Lambda=(\lambda_n)_{n\ge1}$ in $\mathbb N$, we lift it to $\Phi(\Lambda)\subset \mathbb N_0^{(\mathbb N)}$.
It is more practical to geometrically or arithmetically describe the lifted set $\Phi(\Lambda)$
than the set $\Lambda$. That is the reason
why we introduce the lifting map.
When $\Lambda$ is contained in a finitely generated sub-semigroup of $\mathbb{N}$, $\Phi(\Lambda)$ is contained in a
simpler finitely generated semigroup
$\mathbb{N}_0^d$ with some $1\le d<\infty$
($d$ being the number of prime  generators). Here $\mathbb{N}_0^d$ is considered as a sub-semigroup of $\mathbb{N}_0^{(\mathbb{N})}$.

The Khinchin problem that we are considering concerns the action of the multiplicative semigroup of positive integers
$\mathbb{N}$ on the Lebesgue space $(\mathbb{T}, \mathcal{B}(\mathbb{T}), \mathfrak m)$. Indeed, any 
integer $a \in \mathbb{N}$ defines an endomorphism $T_a$ of $\mathbb{T}$ by $$T_a x = ax \mod 1$$ which preserves the Lebesgue measure $\mathfrak m$,  and 
$T_a$ and $T_b$ commute for any $a,b\in\mathbb N$:
$$
  T_a \circ T_b = T_b \circ T_a.
$$  
It is actually an $\mathbb N_0^{(\mathbb N)}$-action on $(\mathbb T,\mathcal B(\mathbb T),\mathfrak m)$, defined by 
$$
     T^\mathbf{n} x = \Psi(\mathbf{n}) x \mod 1.
$$
It is then not surprising that this $\mathbb N$-action on $\mathbb T$ is connected to the 
 translation actions, one on the semigroup $ \mathbb{N}_0^{(\mathbb N)}$ and the other on the group  $\mathbb{Z}^{(\mathbb N)}$, respectively defined by
\[
\tau_S^{\mathbf n}\mathbf m=\mathbf m+\mathbf n
\quad(\mathbf m,\mathbf n\in  \mathbb{N}_0^{(\mathbb N)}),
\qquad
\tau_G^{\mathbf n}\mathbf m=\mathbf m+\mathbf n
\quad(\mathbf n\in \mathbb{N}_0^{(\mathbb N)},\ \mathbf m \in  \mathbb{Z}^{(\mathbb N)}).
\]
Here we use ``S" and  ``G" to distinguish the semigroup translation and the group translation.

Given $\Lambda=(\lambda_k)_{k\ge1}$, we denote the lift of $\lambda_k$ by
\[
\mathbf n_k:=\Phi(\lambda_k)\in\mathbb N_0^{(\mathbb N)},\qquad k\ge1.
\]
The translation $\tau_S$ along with  $(\mathbf{n}_k)$  is a discrete model for the $L^p$-Khinchin problem.
The associated averaging operators are defined by 
\begin{equation}\label{eq:A_S}
A^{\tau_S}_ Ng(\mathbf m):=\frac1N\sum_{k=1}^N
g\bigl(\mathbf m+\mathbf{n}_k), \quad \mathbf{m} \in \mathbb{N}_0^{(\mathbb N)}
\end{equation}
for functions $g: \mathbb{N}_0^{(\mathbb N)} \to \mathbb{C}$. 
Similarly,  the translation $\tau_G$ along   $(\mathbf{n}_k)$  is another discrete model and the  
associated averaging operators are defined by 
\begin{equation}\label{eq:A_G}
A^{\tau_G}_Ng(\mathbf m):=\frac1N\sum_{k=1}^N
g\bigl(\mathbf m+\mathbf{n}_k), \quad \mathbf{m} \in \mathbb{Z}^{(\mathbb N)}
\end{equation}
for functions $g: \mathbb{Z}^{(\mathbb N)} \to \mathbb{C}$.
Our first result, Theorem \ref{thm:Lp-main-criterion} below, describes $L^p$-{Khinchin} sequences by the weak $(p,p)$ maximal inequality for
$A^{\tau_S}_N$'s and also by that for $A^{\tau_G}_N$'s.

\begin{theorem}[Criteria for $L^p$-{Khinchin} sequences using discrete translations]
\label{thm:Lp-main-criterion}
Let \(1\le p<\infty\), and let
\(\Lambda=(\lambda_n)_{n\ge1}\) be a sequence of distinct positive
integers.  The following are equivalent:
\begin{enumerate}
\item[\textup{(i)}] \(\Lambda\) is an \(L^p\)-Khinchin sequence (namely \eqref{eq:WPMI} holds);
\item[\textup{(ii)}] there exists a constant \(C>0\) such that, for every finitely
supported \(g:\mathbb{N}_0^{(\mathbb N)} \to\mathbb C\) and every \(t>0\), we have
\[
\#\left\{\mathbf m\in \mathbb{N}_0^{(\mathbb N)}:
\sup_{N\ge1}\left|
A^{\tau_S}_ Ng(\mathbf{m})
\right|>t\right\}
\le C\cdot t^{-p}\cdot\|g\|_{\ell^p(\mathbb{N}_0^{(\mathbb N)})}^p;
\]
\item[\textup{(iii)}] there exists a constant \(C>0\) such that,
for every finitely supported
\(g: \mathbb{Z}^{(\mathbb N)}\to\mathbb C\) and every \(t>0\),
\[
\#\left\{\mathbf m\in \mathbb{Z}^{(\mathbb N)}:
\sup_{N\ge1}\left|
A^{\tau_G}_Ng(\mathbf{m})
\right|>t\right\}
\le C\cdot t^{-p}\cdot\|g\|_{\ell^p(\mathbb{Z}^{(\mathbb N)})}^p.
\]
\end{enumerate}
Moreover, the optimal constants in \textup{(ii)} and \textup{(iii)} are equal.
\end{theorem}

Theorem \ref{thm:Lp-main-criterion} will be proved as a consequence of Theorem \ref{thm:finite-transfer}, a transference principle on general probability spaces.
Theorem \ref{thm:finite-transfer} could be considered as a Conze-type principle
for semigroup $\mathbb{N}_0^{(\mathbb N)}$-actions and for group $\mathbb{Z}^{(\mathbb N)}$-actions.
Conze's original principle deals with invertible transformations, i.e., $\mathbb{Z}$-actions \cite{Conze1973}.
One of the improvements of Theorem \ref{thm:finite-transfer} with respect to Conze's principle is that it is valid for semigroup actions.
The proof of Theorem \ref{thm:finite-transfer} is based on Calder\'on's transference argument and Avila--Candela's Rokhlin lemma  in \cite{Avila_Candela2016} which is a key point. 

We  view $\tau_S$ and $\tau_G$ as two models. There are other models.  Let
\(\pi:(\mathbb N,\cdot)\to(\mathbb T,+)\) be a semigroup homomorphism.  Define
$\tau_\pi^\mathbf{n}: \mathbb{T}\to \mathbb{T}$ by
$$
\tau_\pi^{\mathbf n}(t)=t+\sum_j n_j\pi(\mathfrak p_j)\pmod 1.
$$
We view $\tau_\pi$ as a circle translation model.

\begin{theorem}[Criterion for $L^p$-Khinchin sequences using circle  homomorphisms] \label{cor:homomorphism-criteria}
Let \(1\le p<\infty\).  Let \(\Lambda=(\lambda_n)_{n\ge1}\subset\mathbb N\) be a sequence of
distinct positive integers.
\begin{enumerate}
\item[\textup{(i)}]  If \(\Lambda\) is an
\(L^p\)-Khinchin sequence, then
there exists \(C>0\) such that, for every semigroup homomorphism
\(\pi:(\mathbb N,\cdot)\to(\mathbb T,+)\), every
\(h\in L^p(\mathbb T)\), and every \(t>0\),
\begin{align}\label{translation-maximal}
\mathfrak m\left\{
y\in\mathbb T:
\sup_{n\ge1}
\left|
\frac1n\sum_{k=1}^n h\bigl(y+\pi(\lambda_k) \bigr)
\right|>t
\right\}
\le
\frac{C}{t^p}\|h\|_{L^p(\mathbb T)}^p .
\end{align}
The constant \(C\) may be chosen independent of \(\pi\).
\item[\textup{(ii)}]  
Conversely, if \eqref{translation-maximal} holds for some injective
semigroup homomorphism \(\pi\), then \(\Lambda\) is an
\(L^p\)-Khinchin sequence.
\end{enumerate}
\end{theorem}

We can construct, in principle, all semigroup homomorphisms  from \((\mathbb N,\cdot)\) to $(\mathbb{T}, +)$, because
every such semigroup homomorphism  is determined by
its values on the primes which can be arbitrary. But, in general, the behavior of  such a homomorphism $\pi$  on $\Lambda$ is not always very clear.
As we will see, the existence of a homomorphism $\pi$ for which  $\pi(\lambda_n)$ admits a limit in $\mathbb{T}$ is an obstruction for  $\Lambda$ being a 
$L^p$-Khinchin sequence (Theorem \ref{cor:Bourgain-homomorphism-obstruction}).  

 Theorem \ref{cor:homomorphism-criteria} is in the spirit of Quas--Wierdl
\cite{Quas_Wierdl2010}, who used the homomorphism
\(n\mapsto\log n\) to study the failure of the \(L^p\)-Khinchin property.
Together with a result of Bourgain (Proposition~10 in \cite{Bourgain1988}), it allows us to get an obstruction for $\Lambda$ being 
  an $L^p$-Khinchin sequence (Theorem \ref{cor:Bourgain-homomorphism-obstruction}). 
Thus we get the following list of non-\(L^p\)-Khinchin sequences:
\begin{enumerate}
\item $(2^n+1)_{n\ge1}$ and $(2^n-1)_{n\ge1}$.  See Proposition~\ref{cor:2n-plus-one-not-khinchin}.
\item {{Fibonacci sequence $(F_n)_{n\ge1}$ defined by
$F_1=F_2=1$, $F_n=F_{n-1}+F_{n-2}$ for $n\ge3$.}}
See Proposition~\ref{cor:fibonacci-not-khinchin}.
\item Any sequence $(q_n)_{n\ge 1}$ of prime numbers (in any order) and the associated product sequence $(q_1q_2\cdots q_n)_{n\ge 1}$. See Proposition~\ref{cor:prime-subsequence-not-khinchin}.
\end{enumerate}

Notice that if two sequences of integers differ for only finitely many  terms, they are $L^p$-Khinchin sequences at the same time. 

Recall that 
\((2^n+1)\) 
is nothing but a translation of \((2^n)\). Thus, we see that the Khinchin property is not preserved by translation.  However,
it is plainly preserved by multiplication: if
\((\lambda_n)\) is an \(L^p\)-Khinchin sequence, so is
\((a\lambda_n)\) for any \(a\in\mathbb N\).

The sequence $(2^n+1)$ shares another pathological property. 
R. Salem and A. S. Zygmund \cite{SZ1948}
proved the central limit theorem (CLT) for the sequence $(\cos(2\pi (2^k+1)x))_{k\ge 1}$. 
But if the cosine  function is replaced by a more general function, the CLT may fail:  Erd\"{o}s and Fortet gave the example $\phi(x)=\cos(2\pi x)+\cos(4\pi x)$  for which 
the limit distribution of the normalized sums $\frac{1}{\sqrt{N}}\sum_{n=1}^N \phi((2^n+1)x)$
is not a Gaussian distribution, but a mixture of Gaussian distributions having characteristic function $t \mapsto \int_0^1 e^{(\cos (\pi y))^2 t^2} dy$.
So, the two sequences $(2^n x \mod 1)$ and $((2^n+1)x \mod 1)$, considered as random processes  with respect to the Lebesgue measure,  are very different. 
Conze and Le Borgne  \cite{CL2011} generalized this example, interpreted as the limiting behavior of certain modified ergodic sums in the framework of dynamical systems. Their generalizations include, e.g., the case of expanding maps, group actions, and chaotic dynamical systems under the assumption of multiple decorrelation. Several examples and applications are given, e.g., for the case of automorphisms of tori and automorphisms of nilmanifolds.

Let us now consider sequences $\Lambda$ contained in a finitely generated sub-semigroup of $\mathbb N$.   
Fix \(d\ge1\) and let \(q_1,\ldots,q_d\ge2\) be pairwise coprime.  Consider the semigroup
\[
S(q_1,\ldots,q_d)
:=
\left\{
q_1^{n_1}\cdots q_d^{n_d}:
(n_1,\ldots,n_d)\in\mathbb N_0^d
\right\}.
\]
The mapping $\phi:S(q_1,\ldots,q_d)\longrightarrow\mathbb N_0^d$ defined by 
\[
\phi(q_1^{n_1}\cdots q_d^{n_d})=(n_1,\ldots,n_d)
\]
is a semigroup isomorphism. 
Notice that, in general, $\phi$ is not a restriction of $\Phi$ defined above. 
But we  also call $\phi$ the lifting map from $S(q_1,\ldots,q_d)$ to $\mathbb N_0^d$.

\begin{theorem}[Criterion for $L^p$-Khinchin sequences in $S(q_1,\ldots,q_d)$]
\label{cor:finite-generated-Lp}
Let \(1\le p<\infty\), and let
\(\Lambda=(\lambda_n)_{n\ge1}\subset S(q_1,\ldots,q_d)\) be a sequence of
distinct numbers.  Then \(\Lambda\) is an \(L^p\)-Khinchin sequence if and
only if there exists \(C>0\) such that, for every finitely supported
\(g:\mathbb Z^d\to\mathbb C\) and every \(t>0\),
\[
\#\left\{\mathbf z\in\mathbb Z^d:
\sup_{N\ge1}\left|
\frac1N\sum_{k=1}^N
g\bigl(\mathbf z+\phi(\lambda_k)\bigr)
\right|>t\right\}
\le C\cdot t^{-p}\cdot\|g\|_{\ell^p(\mathbb Z^d)}^p.
\]
\end{theorem}

Theorem~\ref {cor:finite-generated-Lp} is also a consequence of Theorem \ref{thm:finite-transfer}. If the lifted sequence {\((\phi(\lambda_k))_{k\ge1}\)} satisfies a condition of
Tempelman type, 
Theorem \ref{cor:finite-generated-Lp} implies that $(\lambda_k)$ is {an $L^1$-Khinchin} sequence. Recall that
for subsets \(H_1,H_2\subseteq\mathbb Z^d\), their Minkowski difference is defined by
$$  H_1-H_2:=\{h_1-h_2:h_1\in H_1,\ h_2\in H_2\}.
$$

\begin{theorem}[Tempelman type condition]
\label{cor:difference-set-L1}
A sequence of
distinct integers \(\Lambda=(\lambda_n)_{n\ge1}\)  contained in the semigroup $S(q_1,\ldots,q_d)$
is an \(L^1\)-Khinchin sequence if there exists
a constant \(C>0\) such that, for every \(n\ge1\),
\begin{equation}\label{eq:I-I}
\#(I_n-I_n)\le C\# I_n,
\end{equation}
where $
I_n:=\{\phi(\lambda_k):1\le k\le n\}$. 
\end{theorem}

The sufficient condition \eqref{eq:I-I} can be easily realized for
some 
multiplicative sequences in $S(q_1, \ldots, q_d)$.
Following \cite{FFQQ}, to a given infinite word
\(u=(u_n)_{n\ge1}\in\{1,2,\ldots,d\}^{\mathbb N}\), we associate the so-called
\emph{multiplicative sequence}
\(\Lambda(u)=(\lambda_n)_{n\ge1}\) defined by
\[
  \lambda_n=q_{u_1}q_{u_2}\cdots q_{u_n},
  \qquad n\ge1.
\]
It was proved in \cite{FFQQ} that when \(u\in\{1,2\}^{\mathbb N}\) is the
Thue--Morse sequence defined by the substitution $1\to 12, 2\to 21$, \(\Lambda(u)\) in $S(q_1, q_2)$ is an \(L^1\)-Khinchin sequence.  The proof uses
special features of the Thue--Morse sequence and does not apply to other substitutive sequences.  
Now, as consequences of Theorem \ref{cor:difference-set-L1}, we get the following list of $L^p$-Khinchin sequences:
\begin{enumerate}
\item  \(\Lambda(u)\) is
\(L^1\)-Khinchin for any balanced substitutive sequence
\(u\in\{1,2,\ldots,d\}^{\mathbb N}\) (see Proposition~\ref{prop:balanced-L1}).
{By Adamczewski \cite[Theorem~1(i)]{Adam04}, a primitive substitutive sequence is balanced when the second-largest eigenvalue of its incidence matrix has absolute value strictly less than $1$.}
\item \(\Lambda(u)\) is
\(L^1\)-Khinchin if $(u_n)$ is Sturmian, because Sturmian sequences are balanced. 
\end{enumerate}


As another  application of  Theorem~\ref{cor:difference-set-L1}, we can give another  proof of the following theorem due to
Marstrand \cite{Marstrand} and Nair \cite{Nair1990}.

\begin{theorem}[Marstrand--Nair]\label{cor:marstrand-nair}
Let \(S(q_1,\ldots,q_d)\) be arranged in increasing order. Then this increasing
enumeration is an \(L^1\)-Khinchin sequence.
\end{theorem}

Let \(d\ge2\).  Consider the multiplicative sequence
\(\Lambda(u)=(\lambda_n(u))_{n\ge1}\) in the semigroup
\(S(q_1,\ldots,q_d)\), where \(u_1,u_2,\ldots\) are chosen independently
from the alphabet \(\{1,\ldots,d\}\) according to a probability vector
\(\mathbf r=(r_1,\ldots,r_d)\).
It is proved in \cite{FFQQ} that for any $f \in L^1(\mathbb{T})$, 
for almost every $u \in \{1,2, \ldots,  d\}^{\mathbb{N}}$, 
\begin{align*}
  \lim_{n \to \infty}\frac{1}{n} \sum_{k=1}^n f(q_{u_1} q_{u_2} \cdots q_{u_k} x) = \int_{\mathbb{T}} f \, d \mathfrak{m} \qquad \mathfrak{m}-a.e.
\end{align*}
This result might suggest that, for almost every \(u\), \(\Lambda(u)\) is an \(L^1\)-Khinchin sequence. 
But we prove here that almost surely $\Lambda(u)$ is  not an
$L^p$-Khinchin sequence.

\begin{theorem}\label{thm:random-symbolic-non-Linfty}
  Let \(d\ge2\), and let \(\{1,2,\ldots,d\}^{\mathbb N}\) be equipped
  with the \((r_1,r_2,\ldots,r_d)\)-Bernoulli measure with
  \(r_j>0\) for all \(1\le j\le d\). Then for almost every
  \(u \in \{1, 2, \ldots, d\}^{\mathbb{N}}\), 
  \(\Lambda(u)\) is not an \(L^p\)-Khinchin sequence {for any $1\le p<\infty$}.
\end{theorem}

{To prove Theorem \ref{thm:random-symbolic-non-Linfty}, we  project the involved Bernoulli random walk on $\mathbb{N}_0^d$   to a two-dimensional Bernoulli random walk on $\mathbb{Z}^2$. Then we combine the sweeping-out result of Lacey--Petersen--Wierdl--Rudolph \cite[Theorem~4 and the subsequent remarks]{Lacey1994} with the  transference principle (Theorem \ref{thm:finite-transfer}) to prove the failure of the corresponding weak $(p,p)$ maximal inequality for the translation.}

{\em Organization of the paper.}\ \ 
{In Section~\ref{sec:characterization-weak-pp}, we characterize $L^p$-Khinchin sequences by weak $(p,p)$ maximal inequalities. In Section~\ref{sec:Lp-transference}, a transference principle for measure-preserving $\mathbb N_0^d$-actions and the translation models on $\mathbb N_0^d$ and $\mathbb Z^d$ is proved and an extension to actions of countably generated semigroups and groups is also discussed with semigroup $\mathbb N_0^{(\mathbb N)}$ and 
group $\mathbb Z^{(\mathbb N)}$ as models. Section~\ref{sec:proofs-main-criteria} derives the discrete criterion and circle criterion for $L^p$-Khinchin sequences. Section~\ref{sec:proofs-finite-generated} studies sequences contained in a finitely generated semigroup, and Section~\ref{sec:multiplicative-symbolic} studies multiplicative sequences.}

{\em Acknowledgements.}
The authors would like to thank Christoph Aistleitner and Jean-Pierre Conze for helpful discussions and valuable comments on this work. 

\section{Characterization of \texorpdfstring{\(L^p\)}{Lp}-Khinchin sequences by a weak \((p,p)\) maximal inequality}
\label{sec:characterization-weak-pp}
In this section, we characterize the \(L^p\)-Khinchin sequences  by 
a weak \((p,p)\) maximal inequality for the Khinchin
operators when \(1\le p<\infty\), using the Banach principle and the
Sawyer principle.  

\subsection{{Banach's principle} and Sawyer's principle}
Let us first recall the Banach principle and the Sawyer principle for general operators  
\cite{Banach1926Convergence,Sawyer1966}.
Let \((X,\mathcal B,\mu)\) be a measure space and let
\((L_n)_{n\ge1}\) be a sequence of operators on \(L^p(X)\),  \(1\le p<\infty\). Define the maximal operators $L^*$ and $L^*_N$ by
\[
L^*f(x):=\sup_{n\ge1}|L_n f(x)|,
\qquad
L_N^*f(x):=\sup_{1\le n\le N}|L_n f(x)| \ \ \ (N\ge 1).
\]
We say that $(L_n)$ or \(L^*\) satisfies a weak \((p,p)\) maximal inequality if there
exists a constant \(C>0\) such that
\[
\mu\{x\in X:L^*f(x)>t\}
\le
\frac{C}{t^p}\|f\|_{L^p(X)}^p
\]
for every \(f\in L^p(X)\) and every \(t>0\). Equivalently, for the 
same constant $C$,
\[
\mu\{x\in X:L_N^*f(x)>t\}
\le
\frac{C}{t^p}\|f\|_{L^p(X)}^p
\]
for every $N \ge 1$, every \(f\in L^p(X)\) and every \(t>0\).
The following Banach principle is fundamental {for studying} the almost everywhere convergence of $L_nf(x)$.

\begin{theorem}[{Banach's principle \cite{Banach1926Convergence}}]\label{thm:banach-principle}
Let \((X,\mathcal B,\mu)\) be a measure space and  \(1\le p<\infty\).  Let
\((L_n)_{n\ge1}\) be a sequence of bounded linear operators on \(L^p(X)\). Suppose 
 $L^*f(x) < \infty$ a.e. for every $f \in L^p(X)$.  Then the subspace
\[
\left\{
f\in L^p(X):\lim_{n\to\infty}L_n f(x)\text{ exists for a.e. }x\in X
\right\}
\]
is a closed subspace of \(L^p(X)\).
\end{theorem}

If $(L_n)_{n \ge 1}$ satisfies a weak-$(p,p)$ maximal inequality, then
the finiteness condition on $L^*f(x)$ in Theorem \ref{thm:banach-principle} is satisfied.
The converse assertion is also true under suitable conditions. This is 
Sawyer's strengthened Banach principle. 
Let \(\mathcal F\) be a family of measure-preserving maps, not necessarily
invertible, of a probability space \((X,\mathcal B,\mu)\). We say that  
\(\mathcal F\) is ergodic if every \(\mathcal F\)-invariant measurable set has
measure \(0\) or \(1\). We say that \(\mathcal F\) commutes with \((L_n)\) if,
for every \(T\in\mathcal F\) and every \(f\in L^p(X)\), writing \(g=f\circ T\),
one has
\[
L^*f(Tx)\le L^*g(x)
\qquad\text{for a.e. }x\in X.
\]
Recall that an operator \(L:L^p(X)\to L^p(X)\) is \emph{positive} if \(Lf\ge0\)
a.e. whenever \(f\ge0\) a.e. The following Sawyer principle {strengthens} the Banach principle under extra assumptions.

\begin{theorem}[{Sawyer \cite[Theorem~2]{Sawyer1966}}]\label{Sawyer}
Let \((X,\mathcal B,\mu)\) be a probability space and \(1\le p<\infty\). Let
\((L_n)_{n\ge1}\) be positive bounded linear operators on \(L^p(X)\). Assume
that
\begin{enumerate}
\item[\textup{(1)}] \(L^*f(x)<\infty\) for a.e. \(x\in X\), for every
\(f\in L^p(X)\);
\item[\textup{(2)}] there exists an ergodic family \(\mathcal F\) of
measure-preserving maps of \(X\) which commutes with \((L_n)\). 
\end{enumerate}
Then \(L^*\) satisfies a weak \((p,p)\) maximal inequality.
\end{theorem}

We are going to  use the preceding two principles to study the almost everywhere convergence of Khinchin averages $\Lambda_nf(x)$.
We first mention a result on the 
\(L^p\)-convergence of \((\Lambda_n f)\), which will allow us to identify the limit of \(\Lambda_nf\).

\begin{proposition}[\!\!{\cite[Theorem~3.1]{FFQQ}}] \label{prop:Lp-mean-convergence}
Let \(1\le p<\infty\).  Let $\Lambda=(\lambda_n)_{n\ge1}$ be a sequence of  distinct positive
integers. 
Then for every \(f\in L^p(\mathbb T)\) we have
\[
\frac{1}{n} \sum_{k=1}^n f(\lambda_k x )\longrightarrow \int_{\mathbb T}f\,d\mathfrak m
\qquad\text{in }L^p(\mathbb T).
\]
\end{proposition}

\subsection{\texorpdfstring{\(L^p\)}{Lp}-Khinchin sequences $(1\le p<\infty)$  and weak-\texorpdfstring{\((p,p)\)}{(p,p)} maximal inequality}
\label{section:BPforKS}

We are now ready to present the following characterization {of $L^p$-Khinchin sequences by a weak $(p,p)$ maximal inequality}.
\begin{theorem}\label{thm:Lp-khinchin-maximal}
Let \(1\le p<\infty\). Let
\(\Lambda=(\lambda_n)_{n\ge1}\) be a sequence of distinct positive integers.
Then the following assertions are equivalent:
\begin{enumerate}
\item[\textup{(i)}] \(\Lambda\) is an \(L^p\)-Khinchin sequence;
\item[\textup{(ii)}] there exists a constant \(C>0\) such that, for every
\(f\in L^p(\mathbb T)\) and every \(t>0\),
  \begin{equation}\label{eq:WPMI2}
   \mathfrak{m} \left\{ x \in \mathbb{T} :  \sup_{n \ge 1} |\Lambda_n f(x)| >t \right\} \le \frac{C}{t^p} \|f\|_p^p.
  \end{equation}
\end{enumerate}
\end{theorem}

\begin{proof}
 \((\textup{i})\Rightarrow(\textup{ii})\). Observe that each operator
\(f(x)\mapsto f(\lambda_k x)\) is a linear positive contraction on
\(L^p(\mathbb T)\). Hence the Khinchin operators \((\Lambda_n)_{n\ge1}\) are
positive linear contractions. Moreover, they commute, in Sawyer's sense as recalled above, with
the ergodic endomorphism \(Sx=2x\bmod 1\). If \(\Lambda\) is an
\(L^p\)-Khinchin sequence, then \(\Lambda^*f<\infty\) a.e. for every
\(f\in L^p(\mathbb T)\). A weak \((p,p)\) maximal inequality follows from Sawyer's principle
(see Theorem~\ref{Sawyer}).

 \((\textup{ii})\Rightarrow(\textup{i})\).  Assume \((\textup{ii})\), which implies
\(\Lambda^*f(x)<\infty\) a.e. for every \(f\in L^p(\mathbb T)\).  By the Banach principle, the set of
\(f\in L^p(\mathbb T)\) for which \(\Lambda_n f(x)\) converges a.e. is closed
in \(L^p(\mathbb T)\). Since this convergence holds for continuous functions by Weyl's criterion,
$\Lambda_nf(x)$ converges a.e. for every \(f\in L^p(\mathbb T)\). Proposition~\ref{prop:Lp-mean-convergence}
identifies the \(L^p\)-limit as the constant
\(\int_{\mathbb T}f\,d\mathfrak m\). Therefore the a.e. limit is equal to \(\int_{\mathbb T}f\,d\mathfrak m\).
Thus \(\Lambda\) is an \(L^p\)-Khinchin sequence.
\end{proof}

\section{$L^p$-transference for $\mathbb{N}_0^{d}$-actions and  $\mathbb{N}_0^{(\mathbb N)}$-actions}
\label{sec:Lp-transference}

{
In this section, we  prove a more general result which transfers the weak $(p,p)$ maximal inequality for semigroup actions on a probability space to weak $(p,p)$ maximal inequalities  for translation on the semigroup $\mathbb{N}_0^{(\mathbb N)}$ or for the translation on the group $\mathbb{Z}^{(\mathbb N)}$, where $1\le p<\infty$.}
The proof of Theorem \ref{thm:Lp-main-criterion} will be based on it. This transference also serves as a useful tool in the proofs of some other results in the paper.

The idea of transference comes from Calder\'on \cite{Calderon1968}.
The Calder\'{o}n transference principle is a method for transferring operator bounds from a classical translation setting, such as $\mathbb{R}$ or $\mathbb{Z}$, to operators built from a measure-preserving action on another measure space. 
The point was precisely to move weak and strong $L^p$
 bounds for translation-invariant operators to the corresponding ergodic operators. Coifman–Weiss \cite{CoifmanWeiss1976} later developed this into a general transference framework for group actions.
 The mechanism is usually a Fubini argument plus invariance of measure under the action. The Calder\'{o}n transference principle  allows us to pass from harmonic-analysis or arithmetical inequalities to ergodic-theoretic ones.
 To study our {Khinchin} operators, we adapt the Calder\'{o}n transference principle to actions of the semigroup $\mathbb{N}_0^{(\mathbb N)}$.
 
As we mentioned in the introduction, the action of the multiplicative semigroup of positive integers 
$\mathbb{N}$ on the Lebesgue space $(\mathbb{T}, \mathcal{B}(\mathbb{T}), \mathfrak m)$ is an action by the {semigroup} $\mathbb{N}_0^{(\mathbb{N})}$.
This semigroup also additively acts on the semigroup itself and the group $\mathbb{Z}^{(\mathbb{N})}$.

Each Khinchin operator $\Lambda_n$ is defined by a finite number of integers $\lambda_1, \cdots, \lambda_n$ which are contained in a finitely generated {subsemigroup} of $\mathbb{N}$.
So, 
{every finite collection of Khinchin operators involves only  an action of $\mathbb N_0^d$ for some finite $d$.}
We first consider the semigroup \(\mathbb N_0^d\).  Then there are  three actions to
compare: a measure-preserving \(\mathbb N_0^d\)-action on a probability
space, translation action of \(\mathbb N_0^d\) on itself, 
translation action of \(\mathbb N_0^d\) on the group \(\mathbb Z^d\).
To ensure that we have three equivalent weak $(p,p)$ maximal inequalities, we will need a  finite
tower property for the measure-preserving action. Fortunately, this property is ensured  by  a Rokhlin-type lemma  due to Avila--Candela \cite{Avila_Candela2016} for
free \(\mathbb N_0^{d}\)-actions.  

\subsection{\texorpdfstring{\(\mathbb N_0^d\)}{N0d}-actions,  \(\mathbb N_0^{(\mathbb{N})}\)-actions and notation}
We now consider a general \(\mathbb N_0^d\)-action  (\(d\ge1\)).   Let \(Z\) be an arbitrary set and let
\[
\alpha:\mathbb N_0^d\curvearrowright Z
\]
be an action on \(Z\).  Thus \(\alpha^{\mathbf 0}\) is the identity, $\alpha^\mathbf{n}$ maps $Z$ into $Z$ for every $\mathbf{n} \in \mathbb{N}_0^d$ and
\(\alpha^{\mathbf m+\mathbf n}=\alpha^{\mathbf m}\alpha^{\mathbf n}\) for $\mathbf{m}, \mathbf{n} \in \mathbb{N}_0^d$.
Let $J\ge 1$. Given \(J\) finitely supported weights
\[
\nu_j:\mathbb N_0^d\longrightarrow\mathbb C,
\qquad 1\le j\le J,
\]
we define the corresponding weighted averaging operators, for functions
\(f:Z\to\mathbb C\), by
\begin{equation}\label{eq:general-weighted-average}
A_j^\alpha f(z)
:=\sum_{\mathbf n\in\mathbb N_0^d}
\nu_j(\mathbf n)f(\alpha^{\mathbf n}z),
\qquad 1\le j\le J.
\end{equation}
The associated finite maximal operator is
\[
{A_J^{\alpha,*}f(z):=\max_{1\le j\le J}|A_j^\alpha f(z)|.}
\]
If we are given infinitely many weights $(\nu_j)_{j\ge  1}$, we define the maximal function
$$A^*f(z) = \sup_{j\ge 1}|A_j^\alpha f(z)|.
$$ 

We call  $A^{\alpha, *}f$ the {\em maximal function of $f$ under the action \(
\alpha:\mathbb N_0^d\curvearrowright Z
\) along $(\nu_j)_{j\ge 1}$}, while $A_J^{\alpha,*}f$ is called a {\em truncated maximal function}.

Sometimes, we also consider a general \(\mathbb N_0^{(\mathbb{N})}\)-action  on \(Z\):
\[
\alpha:\mathbb N_0^{(\mathbb{N})}\curvearrowright Z
\]
We consider $\mathbb N_0^d$ as a subsemigroup of \(\mathbb N_0^{(\mathbb{N})}\). Then, when the weight $\nu_j$ is supported on some $\mathbb N_0^d$
($d$ depending on $j$),
we can still formally define $A_j^\alpha f(z)$
 by \eqref{eq:general-weighted-average}.
 If we are given infinitely many weights $(\nu_j)_{j\ge  1}$ each of which is finitely supported, we also define the maximal function
$A^*f(z) = \sup_{j\ge 1}|A_j^\alpha f(z)|
$, called  the {\em maximal function of $f$ under the action \(
\alpha:\mathbb N_0^{(\mathbb{N})}\curvearrowright Z
\) along $(\nu_j)_{j\ge 1}$}.

{\bf Remark.} 
We emphasize that when we consider an $\mathbb N_0^d$-action, we assume that all weights $\nu_j:\mathbb N_0^d\to\mathbb C$ are finitely supported.
When we consider an $\mathbb N_0^{(\mathbb{N})}$-action, we assume that each weight 
 $\nu_j$ is supported on $\mathbb{N}_0^d$ for some $d$ depending on $j$, so that the operators $A_j$ $(1\le j\le J)$, as well as  $A_J^*$, only depend on
 an  $\mathbb N_0^{d_J}$-action for some integer $d_J$ sufficiently large. Therefore the study of  $\mathbb N_0^{(\mathbb{N})}$-actions   essentially reduces to 
 the study of  $\mathbb N_0^{d}$-actions for different $d$'s.
 
 Attention: the weights are not indicated in the notations $A_J^{\alpha, *}f$ and $A^{\alpha, *}f$. 
We should point out that we are interested in the partial behavior of the action  in the union of the supports of the weights, which is not the whole semigroup
$\mathbb N_0^d$ or $\mathbb N_0^{(\mathbb{N})}$ in general. The operators $A_J^{\alpha, *}f$ and $A^{\alpha, *}f$ are determined by the weights, which are usually fixed. 
 
 \color{black}

We will consider three $\mathbb{N}_0^d$-actions. 

{\em {$\mathbb{N}_0^d$-action on $\mathbb{N}_0^d$}}. The first action is the translation of \(\mathbb N_0^d\) on itself:
\[
\tau_S^{\mathbf n}\mathbf m=\mathbf m+\mathbf n
\qquad(\mathbf m,\mathbf n\in\mathbb N_0^d).
\]

{\em {$\mathbb{N}_0^d$-action on $\mathbb{Z}^d$}}. The second is the translation of \(\mathbb N_0^d\) on the
group \(\mathbb Z^d\):
\[
\tau_G^{\mathbf n}\mathbf z=\mathbf z+\mathbf n
\qquad(\mathbf n\in\mathbb N_0^d,\ \mathbf z\in\mathbb Z^d).
\]

{\em $\mathbb{N}_0^d$-action on  \((X,\mathcal B,\mu)\)}. The third is  a  measure-preserving \(\mathbb N_0^d\)-action \(\alpha\) on a
probability space \((X,\mathcal B,\mu)\).  
We say that \(\alpha\) has the
\emph{finite tower property} if, for every finite set
\(\Omega\subset\mathbb N_0^d\), there are pairwise disjoint measurable
sets \((B_{\mathbf m})_{\mathbf m\in\Omega}\), all of the same measure
\(a>0\), such that
\begin{equation}\label{eq:finite-dimensional-tower}
B_{\mathbf m}\subset
(\alpha^{\mathbf n})^{-1}B_{\mathbf m+\mathbf n}
\end{equation}
up to a null set whenever
\(\mathbf m,\mathbf m+\mathbf n\in\Omega\).
We say that the action $\alpha$ is \emph{free} or \emph{aperiodic}  if
\[
\mu\{x:\alpha^{\mathbf m}x=\alpha^{\mathbf n}x\}=0
\qquad(\mathbf m\ne\mathbf n).
\]

We view the translations $\tau_S$ and $\tau_G$ as discrete models of the measure-preserving action $\alpha$.  
If we consider the counting measure on $\mathbb{N}_0^d$ and $\mathbb{Z}^d$ (not a probability measure), 
then $\tau_S$ and $\tau_G$ are formally ``free".

If $\alpha$ is a free $\mathbb N_0^{(\mathbb{N})}$-action on a standard atomless probability space,
its restriction to  $\mathbb N_0^{d}$ is a free $\mathbb N_0^{d}$-action  and hence has the finite tower property by
Proposition~\ref{freetower}.  

\subsection{Free actions have the finite tower property}

The  finite tower property is crucial. It is implied by  Avila--Candela's generalization of the Rokhlin lemma.

Before stating the Avila-Candela-Rokhlin lemma, we introduce some notation.
For \(\mathbf n,\mathbf m\in\mathbb Z^d\), we write \(\mathbf n\le\mathbf m\) if
\(n_j\le m_j\) for every \(1\le j\le d\). For integers \(a\le b\), the notation
\(a\le\mathbf n\le b\) means \(a\le n_j\le b\) for every \(1\le j\le d\).

\begin{lemma}[{Avila--Candela \cite[Theorem~1.2]{Avila_Candela2016}}]
\label{Avila}
Let \((X,\mathcal B,\mu)\) be a standard atomless probability space, and let
\((T^{\mathbf n})_{\mathbf n\in\mathbb N_0^d}\) be a measure-preserving
\(\mathbb N_0^d\)-action on \(X\). Suppose that the action is {free}. Then,
for every \(k\in\mathbb N\) and every \(\varepsilon>0\), there exists a
measurable set \(E\subset X\) such that the sets
\[
(T^{-\mathbf n}E)_{0\le\mathbf n\le k-1}
\]
are pairwise disjoint and
\[
\mu\left(\bigcup_{0\le\mathbf n\le k-1}T^{-\mathbf n}E\right)>1-\varepsilon.
\]
\end{lemma}

\begin{proposition} \label{freetower}
Let \((\alpha^{\mathbf n})_{\mathbf n\in\mathbb N_0^d}\) be a measure-preserving
\(\mathbb N_0^d\)-action on a standard atomless probability space  \((X,\mathcal B,\mu)\). If  the action is free, then it has the finite tower property.
\end{proposition}
\begin{proof}
This is essentially Avila--Candela's tower lemma for an $\mathbb N_0^d$-action (\!\!\cite[Theorem~1.2]{Avila_Candela2016}) which implies the following:
for every integer \(L\ge1\) and every
\(1>\varepsilon>0\), there is a measurable set \(E\) such that the sets
\[
(\alpha^{\mathbf q})^{-1}E,
\qquad \mathbf q\in Q_L:=\{0,1,\ldots,L-1\}^d,
\]
are pairwise disjoint and their union has measure greater than
\(1-\varepsilon (>0)\).  

For a given  finite set \(\Omega\subset\mathbb N_0^d\), choose a large integer \(L\)
such that \(\Omega\subset Q_L\). 
Fix
\[
\mathbf q:=(L-1,\ldots,L-1).
\]
Since $\Omega\subset Q_L$, we have
$\mathbf q-\mathbf m\in Q_L$ for every $\mathbf m\in\Omega$.
For each $\mathbf m\in\Omega$, put
\[
B_{\mathbf m}:=(\alpha^{\mathbf q-\mathbf m})^{-1}E.
\]
These {sets $B_{\mathbf m}$} are pairwise disjoint and have the same positive measure.  For
\(\mathbf m,\mathbf m+\mathbf n\in\Omega\), we have $\mathbf{q} - \mathbf{m}  \in \mathbb{N}_0^d,   \mathbf{q} - \mathbf{m} - \mathbf{n}    \in \mathbb{N}_0^d$
and 
$\alpha^{\mathbf{q}-\mathbf{m}} (B_\mathbf{m})
\subset E$ so that 
$$
    \alpha^{\mathbf{q} - \mathbf{m} - \mathbf{n} } (\alpha^{\mathbf n}B_{\mathbf m}) =
    \alpha^{\mathbf{q} - \mathbf{m}  } (B_{\mathbf m}) \subset E.
$$
Thus we have checked that
\(\alpha^{\mathbf n}B_{\mathbf m}\subset B_{\mathbf m+\mathbf n}\) up
to a null set, and proved the finite tower property.
\end{proof}

\subsection{$L^p$-transference principle for $\mathbb{N}_0^d$-actions and  $\mathbb{N}_0^{(\mathbb{N})}$-actions}

\begin{theorem}[Transference principle for $\mathbb{N}_0^d$-actions on  \(L^p\)]
\label{thm:finite-transfer}
Let \(1\le p<\infty\), and let \(\alpha\) be a measure-preserving
\(\mathbb N_0^d\)-action on \((X,\mathcal B,\mu)\). Suppose that this action has  the finite tower
property.  Then the following weak \((p,p)\) maximal inequalities are
equivalent, with the same constant \(C\):
\begin{enumerate}
\item[\textup{(i)}] for every \(f\in L^p(X)\) and \(t>0\),
\[
\mu\{A_J^{\alpha,*}f>t\}
\le C\cdot t^{-p}\cdot\|f\|_{L^p(X)}^p;
\]
\item[\textup{(ii)}] for every finitely supported
\(\varphi:\mathbb N_0^d\to\mathbb C\) and \(t>0\),
\[
\#\{\mathbf m\in\mathbb N_0^d:
A_J^{\tau_S,*}\varphi(\mathbf m)>t\}
\le C\cdot t^{-p}\cdot\|\varphi\|_{\ell^p(\mathbb N_0^d)}^p;
\]
\item[\textup{(iii)}] for every finitely supported
\(\Phi:\mathbb Z^d\to\mathbb C\) and \(t>0\),
\[
\#\{\mathbf m\in\mathbb Z^d:
A_J^{\tau_G,*}\Phi(\mathbf m)>t\}
\le C\cdot t^{-p}\cdot\|\Phi\|_{\ell^p(\mathbb Z^d)}^p.
\]
\end{enumerate}
If we are given an infinite sequence of weights  $(\nu_j)_{1\le j<\infty}$, we have the same conclusion with 
$A^{\alpha,*}_J$, $A^{{\tau_S},*}_J$, $A^{{\tau_G},*}_J$ replaced by $A^{\alpha,*}$, $A^{{\tau_S},*}$, $A^{{\tau_G},*}$.

The above results remain true for any free   measure-preserving
\(\mathbb N_0^{(\mathbb{N})}\)-action on a  standard atomless probability space  \((X,\mathcal B,\mu)\) under the assumption that the weights
$(\nu_j)_{j\ge1}$ are finitely supported (in the above inequalities,  $\mathbf{m}\in \mathbb{N}_0^d$ 
and $\mathbf{m}\in \mathbb Z^d$ are to be replaced by 
 $ \mathbf{m}\in \mathbb{N}_0^{(\mathbb N)}$ and $ \mathbf{m}\in \mathbb Z^{(\mathbb N)}$ respectively).
\end{theorem}

\begin{proof} {The proof has three steps. The implication \textup{(ii)}$\Rightarrow$\textup{(i)} is the forward Calder\'on transference argument and uses only the invariance of the measure $\mu$.  The proof of  the converse implication \textup{(i)}$\Rightarrow$\textup{(ii)} needs the finite tower property. The  equivalence between \textup{(ii)} and \textup{(iii)} is proved by  a simple discrete translation argument.}

{\bf Step 1.}  $(ii) \Longleftrightarrow (iii)$.  
Given \(\varphi:\mathbb N_0^d\to\mathbb C\), {we extend it to $\Phi:\mathbb Z^d\to\mathbb C$ by putting $\Phi=0$ on $\mathbb Z^d\setminus\mathbb N_0^d$.}
Then we have $$
\|\Phi\|_{\ell^p(\mathbb{Z}^{d})} = \|\varphi\|_{\ell^p(\mathbb{N}_0^{d})}; \quad 
A_j^{\tau_S}\varphi(\mathbf{m}) = A_j^{\tau_G}\Phi(\mathbf{m})  \ \ (\forall \mathbf{m} \in \mathbb N_0^d).
$$ 
Thus
\textup{(iii)} implies \textup{(ii)} with the same constant $C>0$.

Conversely, let \(\Phi\) be finitely supported on \(\mathbb Z^d\).
Then the support of
\(A_J^{\tau_G,*}\Phi\) is finite too.  Choose
\(\mathbf a\in\mathbb N_0^d\) such that 
\[
\mathbf a+\bigl(\operatorname{supp}\Phi\cup\operatorname{supp}(A_J^{\tau_G,*}\Phi)\bigr)\subset\mathbb N_0^d.
\]
Translate back $\Phi$ to get a finitely supported function on $\mathbb{N}_0^d$:
\[
\varphi(\mathbf r):=\Phi(\mathbf r-\mathbf a),
\quad (\forall \mathbf r\in\mathbb N_0^d).
\]
It follows that
\[
\|\varphi\|_{\ell^p(\mathbb N_0^d)}=\|\Phi\|_{\ell^p(\mathbb Z^d)};
\qquad
A_J^{\tau_S,*}\varphi(\mathbf m+\mathbf a)=A_J^{\tau_G,*}\Phi(\mathbf m).
\]
Therefore \textup{(ii)} implies \textup{(iii)} with the same constant $C>0$.

{\bf Step 2.}  $(ii) \Longrightarrow (i)$.
The proof of this implication follows
{Calder\'on's argument}.
{Put
\[
F:=\bigcup_{j=1}^J\operatorname{supp}\nu_j,
\qquad Q_R:=\{0,1,\ldots,R-1\}^d. \quad (R\ge 1).
\]}
Notice that $F$ contains all the weights $\nu_j$ for $1\le j\le J$.

Assume $f\in L^p(X)$. For a fixed point \(x\in X\), we define
the following finitely supported functions on $\mathbb N_0^d$ ($R\ge 1$): 
\[
\varphi_{x,R}(\mathbf r)
:=
\begin{cases}
f(\alpha^{\mathbf r}x),&\mathbf r\in Q_R+F,\\
0,&\text{otherwise}.
\end{cases}
\]
Then, observe that  for \(\mathbf m\in Q_R\) and for every \(1\le j\le J\) we have
\[
A_j^{\tau_S}\varphi_{x,R}(\mathbf m) =\sum_{\mathbf{n}} \nu_j(\mathbf{n}) \varphi_{x,R}(\mathbf m +\mathbf{n})=
\sum_{\mathbf{n}} \nu_j(\mathbf{n}) f(\alpha^{\mathbf{m}+\mathbf{n}}x )
=A_j^\alpha f(\alpha^{\mathbf m}x),
\]
because the sums are taken over $\mathbf{n} \in {\rm supp} \nu_j \subset F$.
{Therefore, for every $x\in X$ and $\mathbf m\in Q_R$,
\[
A_J^{\tau_S,*}\varphi_{x,R}(\mathbf m)
=A_J^{\alpha,*}f(\alpha^{\mathbf m}x),
\]
and hence
\[
\#\{\mathbf m\in Q_R:A_J^{\alpha,*}f(\alpha^{\mathbf m}x)>\lambda\}
\le
\#\{\mathbf m\in\mathbb N_0^d:A_J^{\tau_S,*}\varphi_{x,R}(\mathbf m)>\lambda\}.
\]}
By the assumption of the weak \((p,p)\) inequality (ii)  and the definition of $\varphi_{x, R}$, the right-hand side is bounded by 
\[
C\lambda^{-p}\sum_{\mathbf r\in Q_R+F}|f({\alpha^{\mathbf r}}x)|^p.
\]
Integrating in \(x\),  applying Fubini's theorem and using the $\alpha$-invariance of the measure $\mu$, we get
\[
\mu\bigl(\{x\in X:A^{\alpha,*}_Jf(x)>\lambda\}\bigr) \cdot \#Q_R
\le C\lambda^{-p} {\|f\|_{L^p(X)}^p}\cdot\#(Q_R+F).
\]
Now  we get the desired inequality stated in (i), by dividing both sides  by \(\#Q_R\) and then letting \(R\to\infty\).
Here we used the fact \(\#(Q_R+F)/\#Q_R\to1\) as \(R\to\infty\).

{\bf Step 3.}  $(i) \Longrightarrow (ii)$.
Let \(\varphi:\mathbb{N}_0^d\to\C\) be a finitely supported function. 
Fix \(\lambda>0\). The tail set
\[
E_\lambda:=\{\mathbf{m}\in\mathbb{N}_0^d:A^{\tau_S,*}_J\varphi(\mathbf{m})>\lambda\}
\]
is finite. Consider the finite set in $\mathbb{N}_0^d$:
\[
\Omega={\rm supp} \, \varphi\cup E_\lambda\cup(E_\lambda+F), \quad \text{where}\ \ \ 
F=\bigcup_{j=1}^J {\rm supp} \, \nu_j.
\]
By the finite tower property, there exist pairwise disjoint measurable sets \(B_{\mathbf{m}}\) (\(\mathbf{m}\in\Omega\)), of common measure say \(\gamma>0\), such that
\[
B_{\mathbf{m}}\subset \alpha^{-\mathbf{n}}B_{\mathbf{m}+\mathbf{n}}\quad 
\text{whenever}\ \ \  \mathbf{m},\mathbf{m}+\mathbf{n}\in\Omega.
\]
Define the function \(f:X\to\C\) by
\[
f(x)=\sum_{\mathbf m\in\Omega}\varphi(\mathbf m) {\one_{B_{\mathbf m}}(x)}.
\]
Since the sets \(B_{\mathbf{m}}\) are disjoint and  \(\Omega\) contains the support of \(\varphi\),  we get immediately
\[
\norm{f}_{L^p(X)}^p
=
\gamma \sum_{\mathbf{m}\in\Omega}|\varphi(\mathbf{m})|^p
=
\gamma \norm{\varphi}_{\ell^p(\mathbb{N}_0^d)}^p.
\]
On the other hand, assuming $\mathbf{m} \in E_\lambda$ and \(\mathbf{n}\in F\), we have \(\mathbf{m}+\mathbf{n}\in E_\lambda+F\subset\Omega\). Hence, by the tower property, \(\alpha^{\mathbf{n}} x\in B_{\mathbf{m}+\mathbf{n}}\) for a.e. \(x\in B_{\mathbf{m}}\). Therefore
\[
f(\alpha^{\mathbf{n}} x)=\varphi(\mathbf{m}+\mathbf{n}) \quad \text{for a.e.} \ \  x\in B_{\mathbf{m}}. 
\]
Thus
\[
A^\alpha_{j}f(x)
=
\sum_{\mathbf n\in\mathbb N_0^d}\nu_j(\mathbf n)f({\alpha^{\mathbf n}}x)
=
\sum_{\mathbf{n}\in\mathbb{N}_0^d}\nu_j(\mathbf{n})\varphi(\mathbf{m}+\mathbf{n})
=
A^{\tau_S}_j\varphi(\mathbf{m}) \ \ \ \text{for a.e.}  \ \ x\in B_{\mathbf{m}}. 
\]
Taking the maximum over \(1\le j\le J\), we obtain
\[
\forall \mathbf{m}\in E_\lambda, \quad A^{\alpha, *}_Jf(x)=A^{\tau_S,*}_J\varphi(\mathbf{m})  \ \ \ \text{for a.e.} \ \  x\in B_{\mathbf{m}}.
\]
Hence for each \(\mathbf m\in E_\lambda\), we have {$A_J^{\alpha,*}f(x)>\lambda$} for a.e. \(x\in B_{\mathbf m}\), and therefore
\[
\gamma \#E_\lambda 
\le \sum_{\mathbf{m} \in E_\lambda} \mu\bigl(\{x\in B_\mathbf{m}: A^{\alpha,*}_J f(x)>\lambda\}\bigr) 
\le
\mu\bigl(\{x\in X:A^{\alpha, *}_Jf(x)>\lambda\}\bigr).
\]
Since, by \textup{(i)}, $A_J^{\alpha,*}$ satisfies the weak $(p,p)$ maximal inequality, {we obtain}
\[
\gamma \#E_\lambda
\le
C\lambda^{-p} {\|f\|_{L^p(X)}^p}
=
C\lambda^{-p}\gamma \norm{\varphi}_{\ell^p(\mathbb{N}_0^d)}^p.
\]
{Dividing by $\gamma>0$}, we obtain
\[
\#\{\mathbf{m}\in\mathbb{N}_0^d:A^{\tau_S,*}_J \varphi(\mathbf{m})>\lambda\} 
 = \# E_\lambda \le 
C\lambda^{-p}\norm{\varphi}_{\ell^p(\mathbb{N}_0^d)}^p.
\]

The last assertion is immediate from what we have proved for  $A^{\alpha,*}_J$, $A^{{\tau_S},*}_J$, $A^{{\tau_G},*}_J$. Indeed,
$A^{\alpha,*}_J$, $A^{{\tau_S},*}_J$, $A^{{\tau_G},*}_J$ increase respectively to  $A^{\alpha,*}$, $A^{{\tau_S},*}$, $A^{{\tau_G},*}$
as $J$ increases to infinity.
\end{proof}

{\bf Remark 1.} The implication \textup{(ii)} \(\Rightarrow\) \textup{(i)} holds 
without the assumption of the finite tower property.  

{\bf Remark 2.} 
Similarly we can prove the equivalence between the  following strong \((p,p)\) maximal inequalities:
\[
\|A_J^{\alpha,*}f\|_{L^p(X)}
\le C\cdot\|f\|_{L^p(X)},
\qquad f\in L^p(X),
\]
\[
\|A_J^{\tau_S,*}\varphi\|_{\ell^p(\mathbb N_0^d)}
\le C\cdot\|\varphi\|_{\ell^p(\mathbb N_0^d)},
\qquad \varphi\in\ell^p(\mathbb N_0^d),
\]
\[
\|A_J^{\tau_G,*}\Phi\|_{\ell^p(\mathbb Z^d)}
\le C\cdot\|\Phi\|_{\ell^p(\mathbb Z^d)},
\qquad \Phi\in\ell^p(\mathbb Z^d).
\]
These three inequalities hold with the same constant \(C\).

{\bf Remark 3.}\ 
Recall that in the ergodic theory for one single measure-preserving transformation, Conze’s principle says:
for a fixed sequence of averaging times, a weak $(p,p)$ maximal inequality proved on one aperiodic ergodic system automatically holds on every measure-preserving system \cite{Conze1973}.
Theorem \ref{thm:finite-transfer} generalizes  Conze's principle  to  semigroup $\mathbb{N}_0^{d}$-actions and semigroup $\mathbb{N}_0^{(\mathbb{N})}$-actions, with averages along a given sequence of weights.

{\bf Remark 4.}  \ According to J. P. Conze (personal communication), one could replace a non-invertible system $(X,\mu,T)$ by its natural extension $(\tilde X,\tilde\mu,\tilde T)$.
The natural extension acts on $L^1(\tilde\mu)$, which contains $L^1(\mu)$, as well as all the subspaces $(\tilde T)^{-n} L^1(\mu)$, whose union is dense in $L^1(\tilde\mu)$.
By approximation, this observation should make it possible to show that the weak $(p,p)$ maximal
inequality for $T$ holds
if and only if it holds
for its natural
extension $\tilde{T}$, and the
optimal constants are
the same. This is a  transference principle for the semigroup $\mathbb{N}_0$-action.

\section{Proofs of Theorem~\ref{thm:Lp-main-criterion} and Theorem \ref{cor:homomorphism-criteria}}
\label{sec:proofs-main-criteria}

\subsection{Lattice translation models: Proofs of Theorem~\ref{thm:Lp-main-criterion} }
\begin{proof}
By Theorem \ref{thm:Lp-khinchin-maximal}, 
\(\Lambda=(\lambda_n)\) is an \(L^p\)-Khinchin sequence if and only if the Khinchin operators $(\Lambda_n)$ satisfy the
weak $(p,p)$ maximal inequality \eqref{eq:WPMI2}.
So, we will prove Theorem \ref{thm:Lp-main-criterion} as a consequence of the transference principle stated in Theorem \ref{thm:finite-transfer}, by comparing the weak $(p,p)$ maximal inequalities for three actions.

First, observe that the  Khinchin operators $(\Lambda_N)$ are of the form \eqref{eq:general-weighted-average}, with weights
$$
   \nu_N(\cdot) = \frac{1}{N}  {\one_{\Phi(\Lambda_N)}(\cdot)}
$$
where $\Lambda_N$ also denotes the set of first $N$ terms $\{\lambda_1, \lambda_2, \cdots, \lambda_N\}$ of $\Lambda$ and $\Phi: \mathbb{N} \to \mathbb{N}_0^{(\mathbb{N})}$
is the semigroup isomorphism defined by \eqref{eq:Phi}.
Here the
measure-preserving \(\mathbb{N}_0^{(\mathbb{N})}\)-action
\(\alpha\) acts on $(\mathbb{T}, \mathcal{B}(\mathbb{T}), \mathfrak m)$ as follows
\[
\alpha^{\mathbf n}x:=\Phi^{-1}(\mathbf n)x,
\qquad \forall \mathbf n\in \mathbb{N}_0^{(\mathbb{N})}.
\]
Equivalently, $\alpha^{\Phi(a)} x = a x \mod 1$ for all $a \in \mathbb{N}$. 
Indeed, we have
\[
A_N^\alpha f(x)
=\frac1N\sum_{k=1}^N
f\bigl(\alpha^{\Phi(\lambda_k)}x\bigr)
=\frac1N\sum_{k=1}^Nf(\lambda_kx)
=\Lambda_Nf(x).
\]

{We now complete the proof in three steps.}

{\bf Step 1.}  {\em The above action $\alpha$ is free}.  Indeed, if \(\mathbf m\ne\mathbf n\), then $a\not=b$ where
\(a=\Phi^{-1}(\mathbf m)\) and \(b=\Phi^{-1}(\mathbf n)\).  The equation
\(
\alpha^{\mathbf m} x=\alpha^{\mathbf n} x
\)
in \(\mathbb T
\), namely 
$$
     a x = b x \mod 1
$$
has only finitely many solutions \(x\), all of the form $\frac{k}{a-b} \mod 1$ with $k\in \mathbb{Z}$. 
So, the set $\{x \in \mathbb{T}: \alpha^{\mathbf m} x=\alpha^{\mathbf n} x\}$ has zero Lebesgue measure.

{\bf Step 2.} {\emph{Each restricted $\mathbb N_0^d$-action has the finite tower property.}}
For every $d\ge 1$, the restriction of $\alpha$ to ${\mathbb N}_0^d$ is free. Proposition \ref{freetower} in Section~\ref{sec:Lp-transference}, where the Avila--Candela tower theorem is used, therefore gives the finite tower property for this restricted action.

{\bf Step 3.} {\em Finishing the proof by finite-coordinate reduction.}
{For each fixed $J$, the weights $\nu_1,\ldots,\nu_J$ involve only finitely many coordinates, say $\mathbb N_0^{d_J}$. Apply Theorem~\ref{thm:finite-transfer} to the restricted free action and then let $J\to\infty$. This yields the equivalent weak $(p,p)$ maximal inequalities for the $\alpha$-action and the translations on $\Sinf$ and $\Ginf$, with the same constant.}
Finally, Theorem~\ref{thm:Lp-khinchin-maximal} ensures that  this means $(\lambda_k)$ is an  $L^p$-Khinchin sequence.
\end{proof}

\subsection{Circle translation model: Proof of Theorem \ref{cor:homomorphism-criteria}}

\begin{proof} We prove (i) and (ii) separately.\\
\indent {\em (i).} \
Assume that $\Lambda=(\lambda_n)$ is an $L^p$-{Khinchin} sequence.
By Theorem~\ref{thm:Lp-khinchin-maximal}, the maximal Khinchin
operator $\Lambda^*$  satisfies a
weak \((p,p)\) maximal inequality, or equivalently, a weak $(p,p)$ maximal inequality holds for every truncated maximal Khinchin operator $\Lambda_N^*$  with a common constant \(C>0\).
The operator $\Lambda_N^*$ only involves the finite \(\lambda_1,\ldots,\lambda_N\).
Let $q_1, \cdots, q_{d_N}$ be  all prime divisors occurring in
\(\lambda_1,\ldots,\lambda_N\). Let $S(q_1, \cdots, q_{d_N})$ be the semigroup generated by $q_1, \cdots, q_{d_N}$.
Let
\(
\phi_N:S(q_1, \cdots, q_{d_N})\to \mathbb N_0^{d_N}
\)
be the canonical  semigroup isomorphism
\[
\phi_N\bigl(q_1^{m_1}\cdots q_{d_N}^{m_{d_N}}\bigr)
=
(m_1,\ldots,m_{d_N}) .
\]
For \(1\le k\le N\), denote the $\phi_N$-lift of $\lambda_k$ by 
\[
\boldsymbol\lambda_k^{(N)}:=\phi_N(\lambda_k)\in\mathbb N_0^{d_N}.
\]
Then by the  transference principle (Theorem \ref{thm:finite-transfer}),  the following weak $(p, p)$ maximal inequality  holds for 
the translation on $\mathbb{Z}^{d_N}$ along $(\boldsymbol\lambda_k^{(N)})_{1\le k\le N}$:  for every finitely supported \(g:\mathbb Z^{d_N}\to\mathbb C\) and every
\(t>0\), we have 
\begin{equation}\label{eq:ZdN}
\#\left\{
{\mathbf z}\in\mathbb Z^{d_N}:
\sup_{1\le n\le N}
\left|
\frac1n\sum_{k=1}^n
 g\bigl({\mathbf z}+\boldsymbol\lambda_k^{(N)}\bigr)
\right|>t
\right\}
\le
\frac{C}{t^p}\|g\|_{\ell^p(\mathbb Z^{d_N})}^p
\end{equation}
where  \(C\)  is the above mentioned constant,   independent of \(N\).

Now consider a semigroup homomorphism $\pi : \mathbb{N} \to \mathbb{T}$. 
Denote
\[
\beta_j:=\pi(q_j)\in\mathbb T,
\qquad 1\le j\le d_N.
\]
Then we define a group homomorphism $
\alpha_N:\mathbb Z^{d_N}\longrightarrow\mathbb T
$
by 
\[
\alpha_N(m_1,\ldots,m_{d_N})
:=
m_1\beta_1+\cdots+m_{d_N}\beta_{d_N}   \mod 1.
\]
We further define a measure-preserving \(\mathbb N_0^{d_N}\)-action on
\((\mathbb T,\mathcal{B}(\mathbb{T}), \mathfrak m)\) by
\[
T_N^{\mathbf n}y:=y+\alpha_N(\mathbf n),
\qquad \mathbf n\in\mathbb N_0^{d_N}.
\]
Since \(\pi\) is a semigroup homomorphism, we have 
\[
\alpha_N(\boldsymbol\lambda_k^{(N)})
=
\pi(\lambda_k) \ \ \ \text{for} \ \ 1\le k\le N.
\]
Applying once more the  transference principle, i.e. the implication (ii) $\Longrightarrow$ (i) of Theorem \ref{thm:finite-transfer},  to the action
\((T_N^{\mathbf{n}})_{\mathbf{n} \in \mathbb{N}_0^{d_N}}\),  gives
\begin{equation}\label{eq:TN}
\mathfrak m\left\{
y\in\mathbb T:
\sup_{1\le n\le N}
\left|
\frac1n\sum_{k=1}^n
h\bigl(y+\pi(\lambda_k)\bigr)
\right|>t
\right\}
\le
\frac{C}{t^p}\|h\|_{L^p(\mathbb T)}^p
\end{equation}
for every \(h\in L^p(\mathbb T)\) and every \(t>0\). Since \(C\) is
independent of \(N\), letting \(N\to\infty\) gives the desired maximal
inequality.

{\em (ii).}\  Now we prove the converse of (i) under the assumption that  \(\pi\) is injective. Consider again the action
\(
{\bigl(T_N^{\mathbf m}\bigr)_{\mathbf m\in\mathbb N_0^{d_N}}}.
\)
We first check that this action is free. Indeed, 
assume
\(\mathbf m,\mathbf n\in\mathbb N_0^{d_N}\) and $T_N^{\mathbf m}(y) = T_N^{\mathbf n}(y)$. Then 
\(\alpha_N(\mathbf m)=\alpha_N(\mathbf n)\), namely
\[
\pi\bigl(\phi_N^{-1}(\mathbf m)\bigr)
=
\pi\bigl(\phi_N^{-1}(\mathbf n)\bigr).
\]
Since \(\pi\) is injective and \(\phi_N^{-1}\) is injective too, it follows that
\(\mathbf m=\mathbf n\). So, for \(\mathbf m\neq\mathbf n\), the set 
$\{y\in \mathbb{T}: T_N^{\mathbf m}(y) = T_N^{\mathbf n}(y)\}$ is empty. Hence the action is free.

Then we can apply the implication (i) $\Longrightarrow$ (ii) of Theorem~\ref{thm:finite-transfer} and the assumed inequality \eqref{eq:TN} to get  the weak $(p,p)$ maximal inequality \eqref{eq:ZdN}. 
This is condition \textup{(iii)} of Theorem~\ref{thm:Lp-main-criterion}, with constant $C$.  Hence Theorem ~\ref{thm:Lp-main-criterion}  implies  that $\Lambda$ is an $L^p$-Khinchin sequence.

\end{proof}

\subsection{Examples of non-$L^p$-{Khinchin} sequences}

The following theorem is a corollary of Theorem~\ref{cor:homomorphism-criteria}.
It allows us to confirm the failure of the \(L^p\)-Khinchin property of $\Lambda =(\lambda_k)\subset \mathbb{N}$ through the existence of a homomorphism from $\mathbb{N}$ to $\mathbb{T}$ which is ``stable" 
on the lifted set
of $\Lambda$. 

\begin{theorem}[Obstruction through homomorphisms]
\label{cor:Bourgain-homomorphism-obstruction}
Let \(\Lambda=(\lambda_n)_{n\ge1}\) be a sequence of distinct positive
integers.  Suppose there exists a semigroup homomorphism
$
  \pi:(\mathbb N,\cdot)\to (\mathbb T,+)
$
such that 
$
  \pi(\lambda_n) \to b \ \text{in }\mathbb T,
$
and that \(\pi(\lambda_n)\neq b\) for all sufficiently large \(n\). Then
\(\Lambda\) is not an \(L^p\)-Khinchin sequence for any $1\le p<\infty$.
\end{theorem}

\begin{proof}
We prove it by contradiction. Suppose that  \(\Lambda\) is an
\(L^p\)-Khinchin sequence for some $1\le p<\infty$. 
By discarding finitely many terms, we may assume that
\(
  b_n:=\pi(\lambda_n)\neq b
  \) for all $n\ge1$ and  $b_n\to  b$ in  $\mathbb T$.
 Let us consider the averages for translations on $\mathbb{T}$:
\[
A_Nh(x):=\frac1N\sum_{k=1}^N h(x+b_k).
\]
By {Theorem~\ref{cor:homomorphism-criteria}\textup{(i)}}, the maximal operator $A^*$ satisfies a weak $(p,p)$ maximal inequality. On the other hand,
for every continuous function \(h\in C(\mathbb T)\), the fact \(b_n\to b\) implies that $A_nh(x)$ converges uniformly on $\mathbb{T}$.
As $C(\mathbb{T})$ is dense in $L^p(\mathbb{T})$, the Banach principle implies that $\lim A_nf(x)$ exists almost everywhere for every 
$f\in L^p(\mathbb{T})$. 

However, the sequence \(a_n:=b_n-b\) tends to \(0\); then  by a result due to Bourgain (Proposition~10 in \cite{Bourgain1988}) there exists  a bounded
measurable function \(f\) on \(\mathbb T\) such that 
\[
  \frac1N\sum_{k=1}^N f(x+a_k)
\]
does not converge almost everywhere, namely  \(A_Nh(x)\)  for \(h(x)=f(x-b)\) does not converge almost everywhere, a contradiction.  Thus
\(\Lambda\) is not an \(L^p\)-Khinchin sequence.
\end{proof}

It is easy to see that for any integer $q\ge 2$ and $a\ge 1$, $(aq^n)$ is an $L^1$-Khinchin sequence. 
But when additively perturbed,  it becomes a non-$L^p$-Khinchin sequence for any $1\le p<+\infty$. 

\begin{proposition}\label{cor:2n-plus-one-not-khinchin}
Let $q\ge 2$ and $a\ge 1$ be integers and let $(t_n)$ be a sequence of integers such that $t_n =o(q^n)$ as $n\to +\infty$ and $t_n\not=0$ for large $n$. Then 
$(aq^n +t_n)_{n\ge 1}$ is not an \(L^p\)-Khinchin sequence for any
\(1\le p<\infty\). 

\end{proposition}
\begin{proof}
Define the semigroup homomorphism $\pi:(\mathbb N,\cdot)\longrightarrow(\mathbb T,+)$ by 
\[
\pi(m):=\log_q{m}\ \bmod1.
\]
Set $\lambda_n:=a q^n+ t_n$.  Observe that  $\lambda_n\ge 1$ and  $\pi(\lambda_n)\ne0 \mod 1$ for large $n$,  and
\[
\pi(\lambda_n)
=\log_q a + n+\log_q\left(1+\frac{t_n}{a q^n}\right)\longrightarrow \log_q a
\quad\text{in}\ \mathbb{T}.
\]
Then Theorem~\ref{cor:Bourgain-homomorphism-obstruction} applies.
\end{proof}

Proposition~\ref{cor:2n-plus-one-not-khinchin}  answers a question raised in \cite{FFQQ}: is 
\((q^n+1)\) an $L^p$-Khinchin sequence? The answer is no. 
Proposition~\ref{cor:2n-plus-one-not-khinchin}  applies to $t_n =P(n)$ where $P\in \mathbb{Z}[x]$ is a  non-trivial polynomial of integral coefficients.  So, the sequence $(aq^n+P(n))_{n\ge1}$   is not  
 an $L^p$-Khinchin sequence for any $1\le p<\infty$.  Also, for any two integers $1\le q_1<q_2$, the sequence $(q_1^n +q_2^n)$  is not  an $L^p$-Khinchin sequence for any $1\le p<\infty$. 
 
 Concerning the sequence $(2^n+1)$, we have mentioned  the counterexample of Erd\"{o}s and Fortet in Introduction: for
 $\phi(x)=\cos(2\pi x)+\cos(4\pi x)$, 
$$ \frac{1}{\sqrt{N}}\sum_{n=1}^N \phi((2^n+1)x)$$
converges in distribution to the characteristic function
$$t \mapsto \int_0^1 e^{(\cos (\pi y))^2 t^2} dy$$
which is not a Gaussian distribution, but a mixture of Gaussian distributions. 
 R. Salem and A. S. Zygmund \cite{SZ1948}  (1948) first mentioned it without a full proof, and wrote ``This remark is essentially due to Erd\"{o}s''. 
M. Kac \cite{Kac1949} (1949) later also mentioned the counterexample,  writing ``It thus came as a surprise when simultaneously and independently of each other, Erd\"{o}s and Fortet constructed an example showing that the limit ($\cdots$) need not be Gaussian'', and, in a footnote: ``In Salem and Zygmund [op. cit.] this example is erroneously credited to Erdős alone.''
  For a proof of the limit, Kac refers to his own paper, just an abstract \cite{Kac1947} (1947), and adds the remark ``Details will be given in the paper by Erd\"{o}s, Ferrand, Fortet and Kac''. 
  However, it seems  that neither Erd\"{o}s nor Fortet ever published any paper about this counterexample and the announced paper by Erd\"{o}s, Ferrand, Fortet and Kac apparently has never been published.
  Complete proofs are contained in \cite{AB2010} and \cite{CL2011}.  This history was recounted  by C. Aistleitner in his MathSciNet review on \cite{CL2011}.

\begin{proposition}\label{cor:fibonacci-not-khinchin}
Let \((F_n)_{n \ge 1}\) be the Fibonacci sequence, defined by
\[
F_1=F_2=1,\qquad F_{n+2}=F_{n+1}+F_n .
\]
Then the sequence $(F_n)_{n\ge1}$ is not an \(L^p\)-Khinchin sequence
for any \(1\le p<\infty\).
\end{proposition}

\begin{proof}
  The well-known Binet's formula gives the expression
\[
F_n
=
\frac{\varphi^n-(-\varphi)^{-n}}{\sqrt5}
=
\frac{\varphi^n}{\sqrt5}
\left(1-(-\varphi^{-2})^n\right)
\ \ \ \text{
where }\ \ 
\varphi:=\frac{1+\sqrt5}{2}.
\]
Consider the semigroup homomorphism
\(
\pi:(\mathbb N,\cdot)\longrightarrow(\mathbb T,+)
\)
defined by
\[
\pi(m):=\log_{\varphi} m  \bmod 1 .
\]
Observe that 
\[
\pi(F_n)= \log_{\varphi}\left(1-(-\varphi^{-2})^n\right)- \log_{\varphi}\sqrt5 \mod 1.
\]
Clearly \(\pi(F_n) \to -\log_{\varphi}\sqrt5 \mod 1\) and  \(\pi(F_n) \not= -\log_{\varphi}\sqrt5 \mod 1\) .  Then we conclude with 
Theorem~\ref{cor:Bourgain-homomorphism-obstruction}.
\end{proof}

The same argument applies to many second-order positive linear recurrence  sequences.
For example, let \(a,b\in\mathbb N\), and let
\((u_n)\) be an eventually positive sequence of distinct integers satisfying
\[
u_{n+2}=a u_{n+1}+b u_n.
\]
If
\(
u_n=c\theta^n+c'(\theta')^n
\) with $c>0$, $c'>0$, and \(|\theta'|<\theta\), where \(\theta>1\) and \(\theta'\) are the two roots of the polynomial
\(x^2-ax-b\), then
\[
\log_\theta u_n
=n+\log_\theta c+
\log_\theta\left(1+\frac{c'}{c}\left(\frac{\theta'}{\theta}\right)^n\right).
\]
Then Theorem
\ref{cor:Bourgain-homomorphism-obstruction} applies and \((u_n)\) is not an
\(L^p\)-Khinchin sequence for any $1\le p<\infty$.

Let \(\mathcal P\) be the set of all prime numbers.  We are going to show that 
no sequence of distinct primes is an $L^p$-Khintchine sequence for any  \(1\le p<\infty\). This covers the following result of Marstrand \cite{Marstrand}: the increasing enumeration of all primes is not an
\(L^\infty\)-Khinchin sequence.

\begin{proposition}\label{cor:prime-subsequence-not-khinchin}

Let $(q_n)_{n \ge 1} \subseteq \mathcal{P}$ be 
 a sequence of distinct
primes (not necessarily increasing). Then both $(q_n)_{n \ge 1}$ and $(q_1 q_2 \cdots q_n)_{n \ge 1}$ 
are not \(L^p\)-Khinchin sequences for any
\(1\le p<\infty\).
\end{proposition}

\begin{proof}
Define the semigroup homomorphism
\(\pi: (\mathbb{N}, \cdot) \to (\mathbb{T}, +)\) by
\[
\pi(q_n):=\frac{1}{n}-\frac1{n+1} \ \text{for}\ 
n\ge1; \qquad \pi(p)=0 \ \text{for}\ p\in \mathcal{P}\setminus \{q_n\}.
\]
Then 
\[\pi(q_n)\to0 \ \ \text{ and}\ \ 
\pi(q_1 q_2 \cdots q_n)=\sum_{j=1}^n\pi(q_j)=1-\frac1{n+1} \to 0 \mod 1.
\]
 Then we can apply  Theorem~\ref{cor:Bourgain-homomorphism-obstruction} to conclude. 
\end{proof}

Marstrand proved that the sequence $(n)_{n\ge1}$ is not an $L^\infty$-Khinchin sequence. But
it was proved in \cite{FFQQ} that there exists  a rearrangement $(r_n)$ of  $\mathbb{N}$ which is an \(L^1\)-Khinchin sequence. 
Proposition \ref{cor:prime-subsequence-not-khinchin} shows that it is impossible to rearrange the set of prime numbers $\mathcal{P}$ to get an  
 \(L^p\)-Khinchin sequence. So, regarding the Khinchin property, the two sets $\mathcal{P}$ and $\mathbb{N}$ are radically different.

\section{Proofs of Theorem~\ref{cor:finite-generated-Lp},  Theorem \ref{cor:difference-set-L1} and Theorem \ref{cor:marstrand-nair}}
\label{sec:proofs-finite-generated}

\subsection{Sequences in $S(q_1, \cdots, q_d)$: Proof of Theorem~\ref{cor:finite-generated-Lp}}
\begin{proof}

Consider the
measure-preserving \(\mathbb N_0^d\)-action on \(\mathbb T\) defined by
\[
\alpha^{\mathbf m}x
:=q_1^{m_1}\cdots q_d^{m_d}x\pmod 1,
\qquad \mathbf m=(m_1,\ldots,m_d)\in\mathbb N_0^d.
\]
This action is free. Indeed, if \(\mathbf m\ne\mathbf n\), the pairwise
coprimality of \(q_1,\ldots,q_d\) implies
\(q_1^{m_1}\cdots q_d^{m_d}\ne q_1^{n_1}\cdots q_d^{n_d}\). So, the
set of solutions of the equation \(\alpha^{\mathbf m}x=\alpha^{\mathbf n}x\) is finite, hence of zero Lebesgue measure.
We have thus proved that the action  is free.

Put \(\mathbf n_k:=\phi(\lambda_k)\in\mathbb N_0^d\) for $k\ge1$. {For every $N\ge1$, define a finitely supported weight} on
\(\mathbb N_0^d\) by
\[
\nu_N:=\frac1N\sum_{k=1}^N\delta_{\mathbf n_k}.
\]
The averaging operator along $(\mathbf{n}_k)_{k\ge 1}$ is nothing but  our Khinchin operator:
$$
A_N^\alpha f(x)
:=\frac1N\sum_{k=1}^N f(\alpha^{\mathbf n_k}x)
=\frac1N\sum_{k=1}^N f(\lambda_kx)
=\Lambda_Nf(x).
$$
We also consider the group translation $\tau_G$. Its averaging
operator along $(\mathbf n_k)_{k\ge1}$ is
$$
A_N^{\tau_G}g(\mathbf z)
:=\frac1N\sum_{k=1}^N
g\bigl(\mathbf z+ \mathbf{n}_k\bigr).
$$
{By Theorem~\ref{thm:finite-transfer}, the weak $(p,p)$ maximal inequalities for $A_J^*$ and $A_J^{\tau_G,*}$ are equivalent, with the same constant $C>0$. Hence, if one inequality holds for every $J$, then so does the other; letting $J\to\infty$ gives the corresponding equivalence for $A^*$ and $A^{\tau_G,*}$.}

Suppose first that \(\Lambda\) is \(L^p\)-Khinchin. By
Theorem~\ref{thm:Lp-khinchin-maximal}, there is \(C>0\) such that
\[
{\mathfrak m}\left\{x\in\mathbb T:
\sup_{N\ge1}|\Lambda_Nf(x)|>t\right\}
\le C t^{-p}\|f\|_{L^p(\mathbb T)}^p.
\]
In particular, the same estimate holds with the supremum restricted
to \(1\le N\le J\). Theorem~\ref{thm:finite-transfer} therefore gives
\[
\#\left\{\mathbf z\in\mathbb Z^d:
\max_{1\le N\le J}
\left|\frac1N\sum_{k=1}^N
g\bigl(\mathbf z+\phi(\lambda_k)\bigr)\right|>t\right\}
\le C t^{-p}\|g\|_{\ell^p(\mathbb Z^d)}^p.
\]
Letting \(J\to\infty\) proves the discrete inequality in the
statement, with the same constant \(C\).

Conversely, assume that the discrete inequality in the statement
holds with constant \(C\). Its restriction to \(1\le N\le J\) has
the same constant. Theorem~\ref{thm:finite-transfer} then gives
\[
\mathfrak m\left\{x\in\mathbb T:
\max_{1\le N\le J}|\Lambda_Nf(x)|>t\right\}
\le C t^{-p}\|f\|_{L^p(\mathbb T)}^p.
\]
Letting \(J\to\infty\), we obtain the weak \((p,p)\) maximal
inequality for \(\Lambda^*\). By
Theorem~\ref{thm:Lp-khinchin-maximal}, \(\Lambda\) is
\(L^p\)-Khinchin. Again, the constant \(C\) is unchanged.
\end{proof}

\subsection{Condition of Tempelman type: Proof of Theorem \ref{cor:difference-set-L1}}

 By Theorem~\ref{cor:finite-generated-Lp}, it suffices for us to establish a weak $(1,1)$ maximal inequality for the operators
\[
A_n g(z) := \frac1n\sum_{k=1}^n g\bigl(z+\phi(\lambda_k)\bigr), \quad g\in \ell^1(\mathbb{Z}^d).
\]
Let \(A^*g(z):=\sup_{n\ge1}|A_ng(z)|\).
Assume that  \(g:\mathbb Z^d\to\mathbb C\)  is  finitely supported and
\(t>0\).  Set
\[
E:=\{z\in\mathbb Z^d:A^*g(z)>t\}.
\]

We first show that \(E\) is a finite set.
Fix \(z\in E\) and let \(n(z)\) be the smallest positive integer such
that
\(
 |A_{n(z)} g(z)| >t.
\)
It follows that 
\[
n(z)<
\frac1t\|g\|_{\ell^1(\mathbb Z^d)}.
\]
Thus we get the following cover of $E$: 
\begin{align*}
  E \subseteq \bigcup_{1 \le n < \frac1t\|g\|_{\ell^1(\mathbb Z^d)} } \{ z \in \mathbb{Z}^d: |A_n g(z)| >t\}.
\end{align*}
Since $g$ is finitely supported, each covering set on the right-hand side is finite and
hence $E$ is finite.

Now we estimate the number $\#E$. Recall that $I_n=\{\phi(\lambda_k): 1\le k\le n\}$ and $\phi$ is injective so that $\#I_n = n$. Put
\[
T(z):=z+I_{n(z)}.
\]
Then for $z\in E$, by the definition of $n(z)$,  
\begin{align}\label{eq:Tz-estimate}
\#T(z)=\# I_{n(z)}
<
\frac1t
\left|
\sum_{u\in T(z)}g(u)
\right|.
\end{align}
We now construct  a cover of $E$ as follows. Set
\(G_1:=E\) and choose \(z_1\in G_1\) such that
\[
n(z_1)=\max_{z\in G_1}n(z).
\]
Having chosen \(z_1,\ldots,z_j\), define
\[
G_{j+1}
:=
G_j\setminus
\bigl(z_j+I_{n(z_j)}-I_{n(z_j)}\bigr).
\]
If \(G_{j+1}\neq\varnothing\), choose \(z_{j+1}\in G_{j+1}\) such that
\[
n(z_{j+1})=\max_{z\in G_{j+1}}n(z).
\]
Since \(E\) is finite, this process stops after finitely many steps. We
obtain points \(z_1,\ldots,z_K\in E\) such that $E$ is covered in the {following} way:
\begin{align}\label{eq:E-cover}
  E
\subseteq
\bigcup_{j=1}^K
\bigl(z_j+I_{n(z_j)}-I_{n(z_j)}\bigr).
\end{align}

We claim that the sets \(T(z_1),\ldots,T(z_K)\) are pairwise disjoint.
Indeed, suppose \(i<j\). Since \(z_j\in G_i\) and \(z_i\) was chosen to
maximize \(n(\cdot)\) on \(G_i\), we have
\[
n(z_j)\le n(z_i).
\]
Hence
\[
I_{n(z_j)}\subseteq I_{n(z_i)}.
\]
If \(T(z_i)\cap T(z_j)\neq\varnothing\), then there exist
\(a\in I_{n(z_i)}\) and \(b\in I_{n(z_j)}\subseteq I_{n(z_i)}\) such that
\[
z_i+a=z_j+b.
\]
Therefore
\[
z_j=z_i+a-b
\in z_i+I_{n(z_i)}-I_{n(z_i)}.
\]
But \(z_j\) was not removed when passing from \(G_i\) to \(G_{i+1}\), a
contradiction. Thus the sets \(T(z_j)\) are pairwise disjoint.

By \eqref{eq:E-cover} and the assumption $\#(I_n- I_n) \le C \# I_n$,
\[
\#E
\le
\sum_{j=1}^K
\#\bigl(I_{n(z_j)}-I_{n(z_j)}\bigr)
\le
C\sum_{j=1}^K \# I_{n(z_j)}
=
C\sum_{j=1}^K \#T(z_j).
\]
Using \eqref{eq:Tz-estimate}, we obtain
\[
\#E
<
\frac{C}{t}
\sum_{j=1}^K
\left|
\sum_{u\in T(z_j)}g(u)
\right|.
\]
Since the sets \(T(z_j)\) are pairwise disjoint,
\[
\sum_{j=1}^K
\left|
\sum_{u\in T(z_j)}g(u)
\right|
\le
\sum_{j=1}^K
\sum_{u\in T(z_j)}|g(u)|
\le
\|g\|_{\ell^1(\mathbb Z^d)}.
\]
Therefore
\[
\#\{z\in\mathbb Z^d:A^*g(z)>t\}
\le
\frac{C}{t}\|g\|_{\ell^1(\mathbb Z^d)}.
\]
This is precisely the weak \((1,1)\)  maximal inequality in
Theorem~\ref{cor:finite-generated-Lp}. Hence \(\Lambda\) is an
\(L^1\)-Khinchin sequence.

\subsection{Marstrand--Nair theorem: Proof of Theorem \ref{cor:marstrand-nair}}

We prove the Marstrand--Nair theorem as a corollary of Theorem \ref{cor:difference-set-L1}.
Let \(\Lambda=(\lambda_n)_{n\ge1}\) be the increasing enumeration of
\(S(q_1,\ldots,q_d)\). For \(R\ge1\), put
\[
P_R:=\phi\bigl(S(q_1,\ldots,q_d)\cap[1,R]\bigr) \subset \mathbb{N}_0^d.
\]
Then
\[
P_R=
\left\{(n_1,\ldots,n_d)\in\mathbb N_0^d:
 n_1\log q_1+\cdots+n_d\log q_d\le \log R
\right\}.
\]
The set $P_R$ is a set of lattice points in a simplex in the Euclidean space $\mathbb{R}^d$. The number $\#P_R$ is controlled by the volume of the simplex:
\[
\#P_R\asymp
\frac{(\log R)^d}{\log q_1\cdots\log q_d}.
\]
Moreover, the difference $P_R-P_R$ is contained in a hypercube:
\[
P_R-P_R\subseteq
\prod_{j=1}^d
\left[-\frac{\log R}{\log q_j},\frac{\log R}{\log q_j}\right].
\]
Hence
\[
\#(P_R-P_R) = O((\log R)^d)=O(\#P_R).
\]
Since \(I_n=P_{\lambda_n}\), the condition in Theorem
\ref{cor:difference-set-L1} has been checked. 

\section{Multiplicative sequences}
\label{sec:multiplicative-symbolic}

We now apply the discrete criterion to multiplicative sequences generated by
symbolic dynamics.  The advantage of this viewpoint is that the exponent vector
\(\phi(\lambda_n)\) becomes the position of a path in \(\mathbb N_0^d\).  Thus
word-combinatorial properties of the symbolic sequence, such as balance or
long gaps, become geometric properties of the path, and these properties can be
inserted directly into the discrete maximal criterion.

Fix \(d\ge2\) and pairwise coprime integers \(q_1,\ldots,q_d\ge2\). For
\(u=(u_n)_{n\ge1}\in\{1,2,\ldots,d\}^{\mathbb N}\), define
\begin{equation}\label{multiplicative-sequence}
\Lambda(u)=(\lambda_n)_{n\ge1},
\qquad
\lambda_n:=q_{u_1}q_{u_2}\cdots q_{u_n}.
\end{equation}
Recall the semigroup homomorphism
\[
\phi:S(q_1,\ldots,q_d)\longrightarrow\mathbb N_0^d,
\qquad
\phi(q_1^{m_1}\cdots q_d^{m_d})=(m_1,\ldots,m_d),
\]
and write
\[
\boldsymbol\lambda_n:=\phi(\lambda_n)
=
\sum_{k=1}^n e_{u_k}\in\mathbb N_0^d,
\qquad n\ge1,
\]
where \(e_1,\ldots,e_d\) are the standard basis vectors of
\(\mathbb Z^d\). We also set \(\boldsymbol\lambda_0:=0\).

\subsection{Multiplicative balanced sequences are Khinchin sequences}

For a symbolic sequence \(u\in\{1,2,\ldots,d\}^{\mathbb N}\), let
\(\mathscr L_n(u)\) be the set of all length-\(n\) words occurring in
\(u\). Following \cite{Adam03},
the balance function  of \(u\) is
\[
B_u(n)
:=
\max_{1\le a\le d}\ \max_{w,w'\in\mathscr L_n(u)}
\bigl||w|_a-|w'|_a\bigr|,
\]
where \(|w|_a\) denotes the number of occurrences of the letter \(a\) in
\(w\).  The  symbolic sequence \(u\) is said to be \emph{$C$-balanced} if \(B_u(n)\le C\) for every
\(n\ge1\). We say that \(u\) is \emph{balanced} if it is \(C\)-balanced for some
\(C>0\).

\begin{proposition}\label{prop:balanced-L1}
Let \(u=(u_n)_{n\ge1}\in\{1,2,\ldots,d\}^{\mathbb N}\) be
\(C\)-balanced. Then the multiplicative sequence
\(\Lambda(u)=(\lambda_n)_{n\ge1}\) defined by
\eqref{multiplicative-sequence} is an \(L^1\)-Khinchin sequence.
\end{proposition}

\begin{proof}
For \(r\ge1\), put
\(
I_r:=\{\boldsymbol\lambda_1,\ldots,\boldsymbol\lambda_r\}.
\)
By Theorem~\ref{cor:difference-set-L1}, it suffices to prove
\[
\#(I_r-I_r)=O(\#I_r).
\]
Assume \(m<n\). For a fixed letter  $1\le i\le  d$, denote the number of $i$'s appearing in the word $u_{m+1}\cdots u_n$ by
\[
N_{m,n}(i):=\#\{m+1\le j\le n:u_j=i\}.
\]
Clearly these numbers appear as  the components of the vector
\[
\boldsymbol\lambda_n-\boldsymbol\lambda_m
=
\sum_{j=m+1}^n e_{u_j}
=
\bigl(N_{m,n}(1),\ldots,N_{m,n}(d)\bigr).
\]
That \(u\) is \(C\)-balanced means that for  every
\(m,m'\ge0\),
\[
\forall p\ge 1, \quad \max_{1\le i \le d} \left|N_{m,m+p}(i)-N_{m',m'+p}(i)\right|
\le C.
\]
Therefore, for each fixed $p\ge1$, all vectors $\boldsymbol\lambda_{m+p}-\boldsymbol\lambda_m$ lie in the $\ell^\infty$-ball of radius $C$ centered at  $\boldsymbol\lambda_p-\boldsymbol\lambda_0$.

Since these vectors have integral coordinates,
\begin{equation}
\forall p\ge 1, \quad \#\{\boldsymbol\lambda_{m+p}-\boldsymbol\lambda_m:m\ge0\}
\le
(2C+1)^d .
\label{eq:balanced-block-bound}
\end{equation}

Now write
\[
I_r-I_r=\{0\}\cup\Omega_r\cup(-\Omega_r),
\]
where
\[
\Omega_r:=\{\boldsymbol\lambda_n-\boldsymbol\lambda_m:1\le m<n\le r\}
\]
and \(\Omega_r\) is covered as follows
\[
\Omega_r=\bigcup_{p=1}^{r-1}\Omega_r(p), \quad {\rm where} \ \
\Omega_r(p):=
\{\boldsymbol\lambda_{m+p}-\boldsymbol\lambda_m:1\le m\le r-p\}.
\]
By \eqref{eq:balanced-block-bound},
\(
\#\Omega_r(p)\le(2C+1)^d
\) for
$1\le p\le r-1$. 
Hence
\[
\#\Omega_r
\le
\sum_{p=1}^{r-1}\#\Omega_r(p)
\le
(r-1)(2C+1)^d.
\]
Since \(\#I_r=r\), we obtain
\[
\#(I_r-I_r)
\le
1+2\#\Omega_r
\le
1+2(r-1)(2C+1)^d
= O(\#I_r).
\]
\end{proof}

Substitutive symbolic sequences of Pisot type are balanced; see \cite{Adam03} for
example. This includes the Thue--Morse, Fibonacci, and
Tribonacci substitutive sequences.

\subsection{Bernoulli multiplicative sequences are not $L^p$-Khinchin: Proof of Theorem~\ref{thm:random-symbolic-non-Linfty}}
We will prove this theorem, using the characterization of $L^p$-Khinchin sequences by a weak \((p,p)\) maximal inequality, the transference principle for Khinchin operators, the general transference principle for $\mathbb N_0^d$-actions, and a result due to Lacey et al. \cite[Theorem~4 and the subsequent remarks]{Lacey1994}.

We will compare two actions, the $\mathbb N_0^d$-action associated with the Khinchin problem and the translation on $\mathbb Z^d$.
Fix $u=(u_k)\in \{1,2,\cdots, d\}^\mathbb{N}$. Denote
$$
 ^u\!A_Nf(x) := \frac{1}{N}\sum_{n=1}^N f(q_{u_1}q_{u_2}\cdots q_{u_n} x), \quad f\in L^p(\mathbb T),
$$
and
$$
 ^u\!A_N^{\tau_G}g(\mathbf y) := \frac{1}{N}\sum_{n=1}^N g\left(\mathbf y + \boldsymbol\lambda_n(u)\right), \quad g\in\ell^p(\mathbb Z^d),
$$
where $\boldsymbol\lambda_n(u)= \sum_{j=1}^n e_{u_j} $ is the lift of $q_{u_1}q_{u_2}\cdots q_{u_n} $. According to our notational convention, $^u\!A_N$ (resp. $ ^u\!A_N^{\tau_G}$) is the averaging operator along $(\boldsymbol\lambda_n(u))$ for the first action (resp. the second action).
By Theorem \ref{thm:finite-transfer}, both $(^u\!A_N)_{N\ge1}$ and $(^u\!A_N^{\tau_G})_{N\ge1}$ satisfy a weak $(p,p)$ maximal inequality, at the same time, with a same constant $C>0$.

Before proving  Theorem \ref{thm:random-symbolic-non-Linfty}, we present  the following useful lemma, which shows that projections preserve the weak $(p, p)$ maximal inequality.
 
\begin{lemma}[Weak $(p,p)$ maximal inequality under projection]
\label{lem:quotient-weak-pp}
Let \(1\le p<\infty\), let
\(\pi:\mathbb Z^d\to\mathbb Z^r\) be a surjective homomorphism, and let
\((\mathbf n_k)_{k\ge1}\) be a sequence in \(\mathbb Z^d\). Suppose
that, for every finitely supported
\(f:\mathbb Z^d\to\mathbb C\) and every \(t>0\),
\begin{equation}\label{eq:Zd}
\#\left\{\mathbf x\in\mathbb Z^d:
\sup_{N\ge1}\left|\frac1N\sum_{k=1}^N
f(\mathbf x+\mathbf n_k)\right|>t\right\}
\le C\cdot t^{-p}\cdot\|f\|_{\ell^p(\mathbb Z^d)}^p.
\end{equation}
Then, for every finitely supported
\(g:\mathbb Z^r\to\mathbb C\) and every \(t>0\),
\begin{equation}\label{eq:Zr}
\#\left\{\mathbf y\in\mathbb Z^r:
\sup_{N\ge1}\left|\frac1N\sum_{k=1}^N
g\bigl(\mathbf y+\pi(\mathbf n_k)\bigr)\right|>t\right\}
\le C\cdot t^{-p}\cdot\|g\|_{\ell^p(\mathbb Z^r)}^p.
\end{equation}
\end{lemma}

\begin{proof}
Since $\pi : \mathbb{Z}^d \to \mathbb{Z}^r$ is a surjective homomorphism of abelian groups and 
since $\mathbb{Z}^r$ is a free (hence projective) $\mathbb{Z}$-module, the induced short exact sequence
\[
0 \longrightarrow \ker(\pi) \longrightarrow \mathbb{Z}^d \xrightarrow{\pi} \mathbb{Z}^r \longrightarrow 0
\]
splits. Consequently, there exists an isomorphism
\[
\mathbb{Z}^d \;\cong\; \mathbb{Z}^r \oplus \ker(\pi)
\]
under which $\pi$ acts as the canonical projection onto the first factor, i.e.\ $\pi(\mathbf{y}, \mathbf{h}) = \mathbf{y}$.
So, every $\mathbf n$ can be identified with \((\pi(\mathbf n),\mathbf h)\), where 
$\mathbf h\in\ker\pi$ is uniquely determined. In particular, write
\[
\mathbf n_k=(\pi(\mathbf n_k),\mathbf h_k),
\qquad \mathbf h_k\in\ker\pi.
\]

Fix a finitely supported \(g:\mathbb Z^r\to\mathbb C\), a positive number
\(t>0\) and an integer \(J\ge1\). Let
\[
E_J:=\left\{\mathbf y\in\mathbb Z^r:
\max_{1\le N\le J}\left|
\frac1N\sum_{k=1}^N
g\bigl(\mathbf y+\pi(\mathbf n_k)\bigr)\right|>t\right\}.
\]
Choose a basis of \(\ker \pi\), let \(Q_L\) be the box of side length
\(2L+1\) in this basis, and put
\[
Q'_L = \bigcap_{k=1}^J (Q_L- \mathbf{h}_k). 
\]
Since there are only finitely many translations of $Q_L$ involved in $Q'_L$  for fixed \(J\), we have
\[
\frac{\#Q_L'}{\#Q_L}\longrightarrow1
\qquad\text{as }L\to\infty.
\]

Lift \(g\) to a finitely supported function \(f_L\) on
\(\mathbb Z^d\) by setting
\[
f_L(\mathbf y,\mathbf h)
:=g(\mathbf y)\one_{Q_L}(\mathbf h),
\qquad \mathbf y\in\mathbb Z^r,\quad\mathbf h\in\ker \pi.
\]
\[
\|f_L\|_{\ell^p(\mathbb Z^d)}^p
=\#Q_L\cdot\|g\|_{\ell^p(\mathbb Z^r)}^p.
\]
For \(\mathbf h\in Q'_L\) and for \(1\le N\le J\), 
\begin{align*}
\frac1N\sum_{k=1}^N
f_L\bigl((\mathbf y,\mathbf h)+\mathbf n_k\bigr)
&=\frac1N\sum_{k=1}^N
f_L\bigl(\mathbf y+\pi(\mathbf n_k),
          \mathbf h+\mathbf h_k\bigr)
=\frac1N\sum_{k=1}^N
g\bigl(\mathbf y+\pi(\mathbf n_k)\bigr).
\end{align*}
Consequently,
\[
E_J\times Q'_L
\subseteq
\left\{\mathbf x\in\mathbb Z^d:
\sup_{N\ge1}\left|\frac1N\sum_{k=1}^N
f_L(\mathbf x+\mathbf n_k)\right|>t\right\}.
\]
Applying the assumed weak \((p,p)\) inequality to \(f_L\) gives
\[
\#E_J\cdot\#Q'_L
\le C\cdot t^{-p}\cdot\#Q_L\cdot
\|g\|_{\ell^p(\mathbb Z^r)}^p.
\]
Divide by \(\#Q'_L\) and let \(L\to\infty\). We obtain
\[
\#E_J
\le C\cdot t^{-p}\cdot\|g\|_{\ell^p(\mathbb Z^r)}^p,
\]
uniformly in \(J\). 
Letting \(J\to\infty\) finishes the proof. \end{proof}

Now we are ready to prove Theorem \ref{thm:random-symbolic-non-Linfty}.

\begin{proof}[Proof of Theorem \ref{thm:random-symbolic-non-Linfty}]
We present a proof by contradiction in three steps. Suppose that there exists an event $F\subset \{1,2, \cdots, d\}^\mathbb{N}$ with probability $\mathbb{P}(F)>0$
such that $\Lambda(u)$ is an \(L^p\)-Khinchin sequence  for every $u\in F$. Then $(^u\!A_N)$ satisfies a weak $(p,p)$ maximal inequality, by Theorem \ref{thm:Lp-khinchin-maximal}.

{\bf Step 1.} {\em A weak $(p,p)$ inequality holds on $\mathbb{Z}^2$ for $u\in F$.}\ \  
By Theorem \ref{thm:Lp-main-criterion} and the above assumption, 
the translation operators $(^u\!A^{\tau_G}_N)$ also satisfy a weak $(p,p)$ maximal inequality for $u\in F$. That is to say, the inequality \eqref{eq:Zd} holds with
$\mathbf{n}_k = \boldsymbol\lambda_k(u)$. 
Then,  by  Lemma \ref{lem:quotient-weak-pp}, the inequality \eqref{eq:Zr} holds  with
$\mathbf{n}_k = \boldsymbol\lambda_k(u)$ and $r=2$ (we choose $r=2$).

{\bf Step 2.} {\em Construction of a random walk on $\mathbb{R}$.}\ \ 
Consider the surjective  homomorphism \(\pi:\mathbb Z^d\to\mathbb Z^2\) 
defined by
\[
\pi(e_1)=e_1,
\qquad
\pi(e_i)=e_2\quad(2\le i\le d).
\]
Here we also use $\{e_1, e_2\}$ to denote the canonical basis of the lattice $\mathbb{Z}^2$.  
Since the random variables $\pi(e_{u_k})$ are i.i.d., the projected path
\[
\pi(\boldsymbol\lambda_n)
=\sum_{k=1}^n\pi(e_{u_k})
\]
is a Bernoulli walk on $\mathbb{Z}^2$, one step $e_1$ with probability \(r_1\) and one step $e_2$ with probability
\(1-r_1\). We associate to it a Bernoulli walk on $\mathbb{R}$ with i.i.d. steps of zero mean:
$$
    w_n := \sum_{k=1}^n \xi_k \quad \text{with} \ \ \xi_k = \one_{\{1\}}(u_k) + a \one_{\{2,3, \cdots, d\}}(u_k) 
$$
where $a$ is chosen to be $a=-\frac{r_1}{1-r_1}$ so that $\mathbb{E}\xi_k =0$. 

{\bf Step 3.} {\em Contradiction.} \ \ 
For any measure-preserving flow \((U_t)_{t\in\mathbb R}\) on a probability space $(X, \mathcal{B}, \mu)$, define the
\(\mathbb Z^2\)-action on $(X, \mathcal{B}, \mu)$ by
\[
\beta^{(m,n)}:=U_{m+an}.
\]
Then, in particular,  \(\beta^{\pi(\boldsymbol\lambda_k)}=U_{w_k}\).
According to Lacey--Petersen--Wierdl--Rudolph
\cite[Theorem~4 and Remark~2]{Lacey1994},  for almost every such path $(w_n)=(w_n(u))$ and in particular for some $u\in F$
and every
\(\varepsilon>0\), there exists a measure-preserving probability flow
\((U_t)\) and a measurable set \(E\) with \(\mu(E)<\varepsilon\) such
that
\begin{equation}\label{eq:Lacey}
\mu\left\{x:
\sup_{N\ge1}\frac1N\sum_{n=1}^N
\one_E(U_{w_n}x)>\frac12\right\}=1.
\end{equation}
By Step 1,  for $u\in F$, the translation averages along \((\pi(\boldsymbol\lambda_k(u)))\) on
\(\mathbb Z^2\) satisfy a weak \((p,p)\) maximal inequality with
constant \(C\). Then by the transference principle (Theorem \ref{thm:finite-transfer}), a  weak \((p,p)\) maximal inequality with the same
constant \(C\) holds for the action $\beta$  along $(\pi(\boldsymbol\lambda_k))$. This, together with \eqref{eq:Lacey},  implies 
\[
1\le C\cdot 2^p\cdot\|\one_E\|_{L^p(\mu)}^p
=C\cdot2^p\cdot\mu(E),
\]
leading to a contradiction, if we have chosen  \(\varepsilon<(C\cdot2^p)^{-1}\).

Thus, for each fixed $p<\infty$, we have proved that almost surely $(\lambda_n(u))$ is a
non-$L^p$-Khinchin sequence. Then almost surely $(\lambda_n(u))$ is a
non-$L^p$-Khinchin sequence for every integer $p=1, 2, \cdots$. 
As  $L^q(\mathbb T)$ decreases as $q$ increases, we can conclude that 
almost surely $(\lambda_n(u))$ is a
non-$L^p$-Khinchin sequence for every real number $1\le p<\infty$. 
\end{proof}

\bibliographystyle{plain}

\bibliography{KP2.bib}

\end{document}